\documentclass{amsart}
\usepackage{tikz-cd}
\usepackage{comment}

\usepackage{euscript}
\usepackage{eufrak} 

\usepackage{esint}
\usepackage[abbrev,backrefs]{amsrefs}
\usepackage{amssymb}
\usepackage{algpseudocode}
\usepackage{algorithm}
\usepackage{mathrsfs} 
\usepackage{combelow}


\DeclareMathOperator{\dist}{dist}

\renewcommand{\leq}{\leqslant}
\renewcommand{\geq}{\geqslant}

\newcommand{\F}{\mathcal{F}}

\newcommand{\R}{\mathbb{R}}

\newcommand{\M}{\mathrm{M}}
\newcommand{\MM}{\mathrm{M}}

\newcommand{\scalprod}[2]{\langle{#1},{#2}\rangle}

\newcommand{\eq}[1]{\begin{equation}{#1}\end{equation}}
\newcommand{\mlt}[1]{\begin{multline}{#1}\end{multline}}
\newcommand{\alg}[1]{\begin{align}{#1}\end{align}}

\newcommand{\set}[2]{\{{#1}\mid{#2}\}}
\newcommand{\Set}[2]{\Big\{{#1}\,\Big|\;{#2}\Big\}}

\newcommand{\Eeqref}[1]{\stackrel{\scriptscriptstyle{\eqref{#1}}}{=}}
\newcommand{\Leqref}[1]{\stackrel{\scriptscriptstyle{\eqref{#1}}}{\leq}}

\newcommand{\LseqrefTwo}[2]{\stackrel{\begin{aligned}\scriptscriptstyle\eqref{#1}\\[-1em] \scriptscriptstyle\eqref{#2}\end{aligned}\vspace{-1cm}}{\lesssim}}
\newcommand{\LseqrefThree}[3]{\stackrel{\begin{aligned}\scriptscriptstyle\eqref{#1}\\[-1em] \scriptscriptstyle\eqref{#2}\\[-1em] \scriptscriptstyle\eqref{#3}\end{aligned}}{\lesssim}}
\newcommand{\Lref}[1]{\stackrel{#1}{\leq}}

\newcommand{\Z}{\mathbb{Z}}

\newcommand{\N}{\mathbb{N}}
\newcommand{\Haus}{\mathcal{H}}

\newtheorem{theorem}{Theorem}

\newtheorem{corollary}[theorem]{Corollary}
\newtheorem{defin}{Theorem}[section]
\theoremstyle{definition}
\newtheorem{definition}[defin]{Definition}
\newtheorem{example}[defin]{Example}
\newtheorem{lem}[defin]{Lemma}
\newtheorem{cor}[defin]{Corollary}
\newtheorem{rem}[defin]{Remark}

\newtheorem{st}[defin]{Proposition}

\theoremstyle{remark}
\newtheorem{remark}[defin]{Remark}

\numberwithin{equation}{section}

\begin{document}

\title{On dimension stable spaces of measures}


\author[D. Spector]{Daniel Spector}
\address{Department of Mathematics, National Taiwan Normal University, No. 88, Section 4, Tingzhou Road, Wenshan District, Taipei City, Taiwan 116, R.O.C.}
\curraddr{}
\email{spectda@protonmail.com}
\thanks{}

\author[D. Stolyarov]{Dmitriy Stolyarov}
\address{St. Petersburg State University, Department of Mathematics and Computer Science,14th Line 29b, Vasilyevsky Island, St. Petersburg, Russia, 199178;  
St. Petersburg Department of Steklov Mathematical Institute, Fontanka 27, St. Petersburg, Russia, 191023}
\curraddr{}
\email{d.m.stolyarov@spbu.ru}
\thanks{}

\subjclass[2010]{Primary }

\date{}

\dedicatory{}

\commby{}

\begin{abstract}
In this paper, we define spaces of measures $DS_\beta(\mathbb{R}^d)$ with dimensional stability $\beta \in (0,d)$.  These spaces bridge between $M_b(\mathbb{R}^d)$, the space of finite Radon measures, and $DS_d(\mathbb{R}^d)= \mathrm{H}^1(\mathbb{R}^d)$, the real Hardy space.  We show the spaces $DS_\beta(\mathbb{R}^d)$ support Sobolev inequalities for $\beta \in (0,d]$, while for any $\beta \in [0,d]$ we show that the lower Hausdorff dimension of an element of $DS_\beta(\mathbb{R}^d)$ is at least $\beta$.  

\end{abstract}

\maketitle
\tableofcontents

\section{Introduction}
\subsection{Introduction and Main Results}
The space $M_b(\mathbb{R}^d)$ of finite Radon measures is a somewhat natural choice in the modeling of physical phenomena: it has a simple mathematical structure; the norm on this space can be interpreted as the mass of the physical quantity that a given measure represents; bounded sequences admit weakly-star compact through the identification of this space with the dual of $C_0(\mathbb{R}^d)$.  Two main caveats to its use, however, are its failure to support a strong-type Sobolev inequality and its lack of stability with respect to dimensional estimates.
 
The failure of an analog of Sobolev's $L^p$ inequality \cite{Sobolev} is the assertion that for $\alpha \in (0,d)$ there is no universal constant $C_1=C_1(\alpha,d)>0$ for which
\begin{align*}
\|I_\alpha \mu \|_{L^{d/(d-\alpha)}(\mathbb{R}^d)} \leq C_1 \|\mu\|_{M_b(\mathbb{R}^d)}
\end{align*}
for all $\mu \in M_b(\mathbb{R}^d)$.  Here we use $I_\alpha \mu$ to denote the Riesz potential of a finite Radon measure $\mu \in M_b(\mathbb{R}^d)$, which is defined pointwise by the formula
\begin{align*}
I_\alpha \mu(x):= \frac{1}{\Gamma(\alpha/2)} \int_0^\infty t^{\alpha/2-1} e^{t\Delta} \mu \;dt \equiv \frac{1}{\gamma(\alpha)} \int_{\mathbb{R}^d} \frac{d\mu(y)}{|x-y|^{d-\alpha}},
\end{align*}
where $\Gamma$ is the Euler Gamma function, $e^{t\Delta}$ is the heat semi-group, whose action on a measure can be computed by convolution of the measure with the function $\frac{1}{(4\pi t)^{d/2}} \exp(-|x|^2/4t)$, and $\gamma(\alpha)$ is a suitable normalization constant.  The prevailing wisdom is that one must be content with an estimate in weak-$L^{d/(d-\alpha)}$, the Marcinkiewicz space, see e.g. \cite{Zygmund} or \cite[Eq (14) on p.~120]{S}.

The lack of dimensional stability is the assertion that a sequence $\{\mu_n\} \subset M_b(\mathbb{R}^d)$ for which
\begin{align*}
\| \mu_n \|_{M_b(\mathbb{R}^d)} &\leq C_2,\\
\operatorname*{dim_\mathcal{H}} \mu_n &\geq \beta, \quad \beta \in (0,d],
\end{align*}
need not guarantee\footnote{Consider, for example, an approximation of the identity.}
\begin{align*}
\operatorname*{dim_\mathcal{H}} \mu \geq \beta,
\end{align*}
for any weak-star limit $\mu$ of any subsequence of $\{\mu_n\}$.  Here we use $\operatorname*{dim_\mathcal{H}} \mu$ to denote the lower Hausdorff dimension of $\mu \in M_b(\mathbb{R}^d)$, which can be computed by the formula
\begin{align*}
\operatorname*{dim_\mathcal{H}} \mu &= \sup \left\{ \beta \in (0,d] \colon \mathcal{H}^\beta(E) =0 \implies |\mu|(E)=0 \right\}\\
&= \inf \left\{ \operatorname*{dim_\mathcal{H}} A \colon |\mu|(A)>0\right\}
\end{align*}
where $|\mu|$ denotes the total variation measure of $\mu \in M_b(\mathbb{R}^d)$.  This is a quantification of the singularity of measures, see e.g. \cite[Chapter 10]{Falconer}.


Yet in practice in the modeling of physical phenomena with measures, one finds these drawbacks are often somewhat surprisingly absent in problems under study.  Consider, for example, the equations of magnetostatics:  If one lets $J\colon\mathbb{R}^3\to\mathbb{R}^3$ denote the electric current density and $B\colon\mathbb{R}^3\to\mathbb{R}^3$ the magnetic field, Maxwell's equations for the magnetic field in the static regime are given by the relations
\begin{align*}
\operatorname*{curl}B&=J,\\
\operatorname*{div}B&=0.
\end{align*}
A result pioneered by J. Bourgain and H. Brezis \cite{BourgainBrezis2004,BourgainBrezis2007} and subsequently refined on the Lorentz scale by Felipe Hernandez and the first author \cite{HS} and independently by the second author \cite{Stolyarov} asserts the existence of a universal constant $C_3>0$ such that
\begin{align}\label{HS}
\|B\|_{L^{3/2,1}(\mathbb{R}^3;\mathbb{R}^3)} \leq C_3\|J\|_{M_b(\mathbb{R}^3;\mathbb{R}^3)}.
\end{align}
Thus despite its failure for general measures, a strong-type Sobolev inequality persists for this model.

As for dimensional stability, if one has a bounded sequence of compatible electric current densities $\{J_n\} \subset M_b(\mathbb{R}^3;\mathbb{R}^3)$, i.e.
\begin{align*}
\|J_n\|_{M_b(\mathbb{R}^3;\mathbb{R}^3)}&\leq C_4,\\
\operatorname*{div}J_n&=0,
\end{align*}
then if $J$ is any weak-star limit of the sequence $\{J_n\}$, necessarily
\begin{align}\label{W}
\operatorname*{dim_\mathcal{H}} J \geq 1,
\end{align}
as \eqref{W} holds for all $J \in M_b(\mathbb{R}^3;\mathbb{R}^3)$ for which $\operatorname*{div}J=0$ in the sense of distributions; for such charges~\eqref{W} follows either from Smirnov's decomposition theorem~\cite{Smirnov} or from the Fourier analytic approach of~\cite{RW} (see also \cite{GHS}).  The Sobolev inequality \eqref{HS} and dimensional stability \eqref{W} rely in a crucial way on the fact that $J$ are not arbitrary vector measures but also satisfy $\operatorname*{div}J=0$.  

The discovery of Sobolev inequalities and dimensional estimates for measures with additional structure has a long history, which in several dimensions can be dated to at least the late 1950s.  The first result which should be mentioned is naturally the extension of Sobolev's inequality due to E. Gagliardo \cite{Gagliardo}, L. Nirenberg \cite{Nirenberg}, Federer and Fleming \cite{FF}, and Maz'ya \cite{mazya}, who independently pioneered Sobolev inequalities in this scaling.  The proofs of Federer and Fleming \cite{FF} and Maz'ya \cite{mazya} rely on the coarea formula, which also implies that for any $ Du \in M_b(\mathbb{R}^d;\mathbb{R}^d)$ one has
\begin{align*}
\operatorname*{dim_\mathcal{H}} Du \geq  d-1.
\end{align*}
As one has examples for which there is equality, e.g.~the weak derivative of the characteristic function of a ball, this shows the dimension of this space is $d-1$.
At essentially the same time one has the work of E. Stein and G. Weiss \cite{Stein-Weiss} which simultaneously achieved different Sobolev inequalities and dimensional estimates.  In particular, in their paper \cite{Stein-Weiss} generalizing the notion of conjugate harmonic functions to more than one dimension, Stein and Weiss introduced
\begin{align*}
\mathrm{H}^1(\mathbb{R}^d) := \{ f \in L^1(\mathbb{R}^d) \colon \nabla I_1 f \in L^1(\mathbb{R}^d;\mathbb{R}^d) \},
\end{align*}
the real Hardy space in several variables.  For this space they proved a number of results.  Concerning an analog of Sobolev's inequality, they demonstrated\footnote{A Lorentz refinement of this inequality is due to Dorronsorro \cite{Dorronsoro} (see also \cite{Spector} for a Lorentz inequality for potentials of gradients, \cite{Adams:1988} for the stronger trace inequality and \cite{Ponce-Spector,RSS,Spector-NA, Spector-PM, Stolyarov-1} for the details of this argument). This inequality and standard harmonic analysis can be shown to imply the inequality \eqref{HS} for $J_i \in \mathrm{H}^1(\mathbb{R}^d)$.} that for $\alpha \in (0,d)$ one has the existence of a universal constant $C_5=C_5(\alpha,d)>0$ such that
\begin{align*}
\|I_\alpha f \|_{L^{d/(d-\alpha)}(\mathbb{R}^d)} \leq C_5 \|f\|_{\mathrm{H}^1(\mathbb{R}^d)}
\end{align*}
for all $f \in \mathrm{H}^1(\mathbb{R}^d)$.  

Concerning dimensional estimates, Stein and Weiss' result can be shown to imply that
\begin{align*}
\operatorname*{dim_\mathcal{H}} f = d,
\end{align*}
for all $f \in \mathrm{H}^1(\mathbb{R}^d)$.  From this one deduces a dimensional stability result of order $d$, as for any sequence which satisfies
\begin{align*}
\|f_n\|_{\mathrm{H}^1(\mathbb{R}^d)}&\leq C_6,
\end{align*}
one has
\begin{align*}
\operatorname*{dim_\mathcal{H}} f = d
\end{align*}
for any weak-star limit $\mu=f\;dx$ of any subsequence of $\{f_n\}$, see also Uchiyama's result \cite{Uchiyama} which establishes such dimension estimates for more general Hardy spaces. 

The question of Sobolev inequalities for $p=1$ and dimensional estimates follows a sort of parallel development from this point, though the two are of course closely related.  Toward Sobolev inequalities, one has the Korn--Sobolev embedding of M.J. Strauss \cite{Strauss} and the work of J. Bourgain and H. Brezis \cite{BourgainBrezis2004,BourgainBrezis2007}, the latter of which led to the systematic treatment of inequalities for closed subspaces of vector-valued measures by J. Van Schaftingen \cite{VanSchaftingen_2013}, who among other results proved that for $\alpha \in (0,d)$ one has
\begin{align*}
\|I_\alpha F \|_{L^{d/(d-\alpha)}(\mathbb{R}^d;\mathbb{R}^l)} \leq C_7 \|F\|_{M_b(\mathbb{R}^d;\mathbb{R}^l)}
\end{align*}
for all $F \in M_b(\mathbb{R}^d;\mathbb{R}^l)$ such that $LF=0$ for some homogeneous constant coefficient linear differential operator $L$ if and only if $L$ is cocancelling. We recall that an operator $L(D)$ as above is said to be \emph{cocancelling} when
	$$
	\bigcap_{\xi\in\mathbb{R}^d}\operatorname*{ker} L(\xi)=\{0\}.
	$$
 Here we identify the differential operator $L$ with its symbol.
 
 Van Schaftingen's result was subsequently refined on the Lorentz and Besov scales \cite{HS, SpectorHernandezRaita2021, Stolyarov} and in certain cases a trace inequality has been obtained \cite{RSS}, see also \cite{SSVS,SVS, VanSchaftingen_2004,VanSchaftingen_2004_ARB,VanSchaftingen_2010,VanSchaftingen_2013, VanSchaftingen_2015} for related estimates in this spirit and \cite{GS,GS1, MS, Chen-Spector, VS} for dual results.  The question of the validity of trace inequalities in the spirit of \cite{Adams:1988,RSS} presents a significant open area in the theory of Sobolev inequalities in this scaling.  

The study of the problem of dimensional estimates can be said to be in a similar state to the question of trace inequalities, as despite numerous result in this direction \cite{AW,ARDPHR,AAR,Raita_report,DS,RA, Stol-Woj}, the problem of the computation of the dimension of the singular set of a given closed subspace $X \subset M_b(\mathbb{R}^d;\mathbb{R}^l)$ has not been fully resolved.  More precisely, let~$L$ be a homogeneous linear differential operator with constant coefficients (or more generally a pseudo-differential operator with constant coefficients) and $X$ be defined by the distributional constraint $L$:
\begin{equation}\label{SpaceX}
    X = \{\mu \in M_b(\mathbb{R}^d;\mathbb{R}^l)\colon  L\mu = 0\}.   
\end{equation}
Then one defines the dimension of $X$ as
\begin{align*}
\kappa:= \inf_{\nu \in X} \operatorname*{dim_{\mathcal{H}}} \nu,
\end{align*}
and the question is to compute $\kappa(X)$.

This question was studied for the first time in~\cite{RW} (in the context of more general Fourier restrictions); the main result says that~$\kappa \geq 1$ provided~$X$ does not contain vectorial delta-measures. In the setting~\eqref{SpaceX}, this is equivalent to~$L$ being cocancelling. The paper~\cite{ARDPHR} gives a bound from below on~$\kappa$ formulated in terms of the $k$-wave cone of~$L$. The latter cone is a purely algebraic object computed from the symbol of~$L$, and the lower bound for~$\kappa$ obtained in this way is always an integer. In the cases where the estimates of~\cite{ARDPHR} are sharp, the measures in~$X$ having the smallest possible dimension are related to flat measures, i.e. Hausdorff measures on affine subspaces. The bounds of~\cite{ARDPHR} are not sharp, see~\cite{RA} for an example along with a sharper result for this example.   In a discrete version of the dimension problem resolved in~\cite{ASW}, it was shown that the extremal measures might be fractal.  Thus one wonders whether this is the case in the continuous analog.  Therefore finally, the state of art in the dimension problem is that it is still unknown whether~$\kappa$ is always an integer, which should be settled, along with a canonical way to compute it in terms of the operator $L$.

The starting point of this paper is the observation that in contrast to the dimension zero case -- the space of finite Radon measures -- and the dimension $d$ case -- the real Hardy space -- the literature concerning Sobolev inequalities around the space of measures and dimensional estimates all focus on the vector-valued setting, in the spirit of the inequality \eqref{HS}.  It therefore seems desirable to introduce scalar spaces of measures with dimensional stability $\beta \in (0,d]$.  To this end we first define the notion of a $\beta$-atom:
\begin{definition}[$\beta$-atom]
For $\beta \in (0,d]$,  we say that $a \in M_b(\mathbb{R}^d)$ is a $\beta$-atom if there exists a cube $Q\subset \mathbb{R}^d$ with sides parallel to the coordinate axes such that
\begin{enumerate}
\item \label{supp}$\operatorname*{supp} a \subset Q$;
\item \label{ave} $a(Q) = 0$;
\item $\sup_{x \in \mathbb{R}^d, t>0} |t^{(d-\beta)/2} e^{t\Delta} a (x)|  \leq \frac{1}{l(Q)^\beta}\label{linfty} $;
\item \label{l1} $|a|(\mathbb{R}^d) \leq 1$.
\end{enumerate}
\end{definition}

We next define atomic spaces with dimensional stability $\beta$:
\begin{definition}[$\beta$-atomic space]
 Let $\beta \in [0,d]$.  Define the atomic space of dimension $\beta$ by
\begin{align*}
DS_\beta&(\mathbb{R}^d):=\\ 
&\left\{\mu \in M_b(\mathbb{R}^d)\colon \mu = \lim_{n\to \infty}\sum_{i=1}^{n} \lambda_{i,n} a_{i,n},    \text{ $a_{i,n}$ $\beta$-atoms, } \limsup_{n \to \infty} \sum_{i=1}^n |\lambda_{i,n}|<+\infty \right\}
\end{align*}
Here the convergence is intended weakly-star as measures,
\begin{align*}
\int \varphi \;d\mu= \lim_{n \to \infty}  \sum_{i=1}^{n} \lambda_{i,n} \int \varphi \;da_{i,n},
\end{align*}
for all $\varphi \in C_0(\mathbb{R}^d)$, and the space is equipped with the norm
\begin{align*}
\|\mu \|_{DS_\beta(\mathbb{R}^d)}:= \inf \left\{\limsup_{n \to \infty} \sum_{i=1}^n |\lambda_{i,n}| \colon \mu = \lim_{n\to \infty}\sum_{i=1}^{n} \lambda_{i,n} a_{i,n} \right\}.
\end{align*}
\end{definition}

We are now prepared to state the main results of this paper.  We begin with the Sobolev inequalites.  Our first result asserts that the space $DS_\beta(\mathbb{R}^d)$ supports a Besov--Lorentz scale Sobolev inequality for any $\beta \in (0,d]$.
\begin{theorem}\label{strong-type}
Let $\beta \in (0,d]$.  For each $\alpha \in (0,d)$ there exists a constant $C=C(\alpha,\beta,d)>0$ such that
\begin{align*}
\left\|I_\alpha \mu \right\|_{\dot{B}^{0,1}_{d/(d-\alpha),1}(\mathbb{R}^d)} \leq C \|\mu\|_{DS_\beta(\mathbb{R}^d)}
\end{align*}
for all $\mu \in DS_\beta(\mathbb{R}^d)$.
\end{theorem}
Our second result asserts that the space $DS_\beta(\mathbb{R}^d)$ supports a trace Sobolev inequality for $\beta \in (0,d]$ at any dimension strictly greater than $d-\beta$ (see \cite{RSS} for a detailed discussion of trace inequalities).
\begin{theorem}\label{dimension_theorem}
Let $\beta \in (0,d]$.  For each $\alpha \in (d-\beta,d)$ there exists a constant $C=C(\alpha,\beta,d)>0$ such that
\begin{align*}
\left\|I_\alpha \mu \right\|_{L^1(\nu)} \leq C \|\mu\|_{DS_\beta(\mathbb{R}^d)}
\end{align*}
for all Radon measures $\nu$ such that $\nu(B(x,r))\leq C'r^{d-\alpha}$ and all $\mu \in DS_\beta(\mathbb{R}^d)$.
\end{theorem}
The preceding theorem need not hold at the endpoint $d-\beta$, see \cite{RSS}.  Our third result asserts that with a fractional maximal function in place of the Riesz potential such a result holds in the endpoint case.
\begin{theorem}\label{dimension_theorem-fractional-maximal}
Let $\beta \in (0,d]$.  For each $\alpha \in [d-\beta,d)$ there exists a constant $C=C(\alpha,\beta,d)>0$ such that
\begin{align*}
\left\|\sup_{t>0} t^{\alpha/2} |e^{t\Delta} \mu| \right\|_{L^1(\nu)} \leq C \|\mu\|_{DS_\beta(\mathbb{R}^d)}
\end{align*}
for all Radon measures $\nu$ such that $\nu(B(x,r))\leq C'r^{d-\alpha}$ and all $\mu \in DS_\beta(\mathbb{R}^d)$.
\end{theorem}

Let us note that in combination with the atomic decomposition proved in \cite{HS}, Theorem \ref{dimension_theorem-fractional-maximal} yields an $L^1$ endpoint of Sawyer's one weight inequality \cite{Sawyer1} (see also \cite{Sawyer2} for the two weight case, which includes the one weight case and is more easily accessible):
\begin{corollary}\label{endpoint_sawyer}
For each $\alpha \in [d-1,d)$ there exists a constant $C=C(\alpha,d)>0$ such that
\begin{align*}
\left\|\sup_{t>0} t^{\alpha/2} |e^{t\Delta} F| \right\|_{L^1(\nu)} \leq C \|F\|_{M_b(\mathbb{R}^d;\mathbb{R}^d)}
\end{align*}
for all Radon measures $\nu$ such that $\nu(B(x,r))\leq Cr^{d-\alpha}$ and all $F \in M_b(\mathbb{R}^d;\mathbb{R}^d)$ which satisfy $\operatorname*{div}F=0$ in the sense of distributions.
\end{corollary}

\begin{remark}\label{mmz}
When $F=\nabla u$, an analogue of the preceding corollary is the assertion that for any $\alpha \in [1,d)$ there is a constant $C=C(\alpha,d)>0$ such that
\begin{align*}
\left\|\sup_{t>0} t^{\alpha/2} |e^{t\Delta} \nabla u| \right\|_{L^1(\nu)} \leq C \|\nabla u\|_{M_b(\mathbb{R}^d;\mathbb{R}^d)}
\end{align*}
for all Radon measures $\nu$ such that $\nu(B(x,r))\leq Cr^{d-\alpha}$.  This follows from the bound
\begin{align}\label{equivalent-maximal}
| t^{1/2}e^{t\Delta} \nabla u| = | t^{1/2} \nabla e^{t\Delta}  u| \lesssim \mathcal{M}(u),
\end{align}
and David R. Adams' extension \cite[Theorem B]{Adams:1988} of the work of Maz'ya \cite[Teopema 4 on p.~165]{maz_trace} and Meyers and Ziemer \cite{Meyers-Ziemer}.  Here $\mathcal{M}$ denotes the Hardy--Littlewood maximal function, see \cite[p.~4]{S} for the definition and \cite[Theorem 2 on p.~62]{S} for a proof of the inequality \eqref{equivalent-maximal}.
\end{remark}

The preceding Sobolev inequalities follow easily from a well-prepared definition of the atomic spaces, while we next turn our attention to dimension estimates, which require a little more work. The idea is that the left-hand-side of Theorem \ref{dimension_theorem-fractional-maximal} contains the essential information to obtain a dimensional estimate for the space $DS_\beta(\mathbb{R}^d)$.  To see this,  let us first consider a simplified setting.
For $\gamma \in [0,d]$ define the dyadic fractional maximal function
\begin{align*}
\mathcal{M}^{Q_0}_{\gamma}(\mu)(x):=\sup_{Q \in \mathcal{D}(Q_0)} \chi_Q(x) \frac{|\mu(Q)|}{\;\;l(Q)^{d-\gamma}},
\end{align*}
where $\mathcal{D}(Q_0)$ denotes the set of all dyadic cubes generated by some fixed cube $Q_0\subset \mathbb{R}^d$, i.e. $Q \in \mathcal{D}(Q_0)$ if and only if $Q=2^{-k}(\bold{n}+Q_0)$ for some $k \in \mathbb{Z}$ and $\bold{n} \in \mathbb{Z}^d$.  Note that for $Q \in \mathcal{D}(Q_0)$, it is not necessarily the case that $Q \subset Q_0$, as we consider the entire lattice generated by $Q_0$.

Then the essential idea of dimension estimates for the spaces $DS_\beta(\mathbb{R}^d)$ is contained in
\begin{theorem}\label{dimension simplified}
Suppose $\mu$ is a locally finite Radon measure in~$\mathbb{R}^d$ such that 
\begin{align}\label{L1}
\int_Q \mathcal{M}^{Q_0}_{d-\beta}(\mu) \;d\mathcal{H}_\infty^\beta <+\infty
\end{align}
for any cube $Q\subset \mathbb{R}^d$.  Then $\dim_{\mathcal{H}} \mu \geq \beta$.
\end{theorem}
\noindent
Here the integral is intended in the sense of Choquet, i.e.~for an everywhere defined non-negative function $f\colon\mathbb{R}^d \to [0,\infty]$,
\begin{align*}
\int_Q f \;d\mathcal{H}_\infty^\beta := \int_0^\infty \mathcal{H}_\infty^\beta(\{x \in Q \colon f>t\})\;dt,
\end{align*}
see also \cite{Adams:1988, Ponce-Spector-2, Ponce-Spector-1} for further results and discussion, while $\mathcal{H}_\infty^\beta$ denotes the spherical Hausdorff content, defined on a set $E \subset \mathbb{R}^d$ via
\begin{align*}
\mathcal{H}_\infty^\beta(E):= \inf \left\{ \sum_{i=1}^\infty \omega_\beta r_i^\beta \colon E \subset \bigcup_{i=1}^\infty B(x_i,r_i)\right\},
\end{align*}
and $\omega_\beta= \pi^{\beta/2}/\Gamma(\beta/2+1)$ (for $\beta=k \in \mathbb{N}$, $\omega_k$ is the volume of the unit ball in $\mathbb{R}^k$).

Note that for $\mu \in DS_\beta(\mathbb{R}^d)$ the conclusion of Theorem \ref{dimension_theorem-fractional-maximal} is locally slightly weaker than the condition \eqref{L1}, and as with the bounds for the Hardy-Littlewood maximal function, one does not expect \eqref{L1} to hold.  Nonetheless, it is possible to show that a slightly stronger analogue of Theorem \ref{dimension_theorem-fractional-maximal} implies the desired dimensional estimate.  To state this stronger analogue, let us introduce a fractional grand maximal function:  For a distribution $f\in \mathcal{S}'(\R^d)$ and $\gamma \in [0,d]$ we define
\begin{align*}
\M_{\F,\gamma}f(x) := \sup_{\Phi \in \F, \|\Phi\|_{\F}\leq 1} t^{\gamma}|f\ast\Phi_t(x)|,
\end{align*}
where $\mathcal{F}$ is a finite collection of Schwartz seminorms, see \cite[p.~90]{Stein93}.  The asserted strengthening of Theorem \ref{dimension_theorem-fractional-maximal} is then
\begin{theorem}\label{maximal bound dsbeta}
There exists a constant $C=C(\beta,d)>0$ such that   
\begin{align*}
\|\M_{\F,d-\beta} \mu\|_{L^1(\Haus_\infty^\beta)} \leq C\|\mu\|_{DS_\beta(\mathbb{R}^d)}
\end{align*}
for all $\mu \in DS_\beta(\mathbb{R}^d)$.
\end{theorem}
\noindent
Here we recall that $L^1(\Haus_\infty^\beta)$ denotes the Banach space of $\Haus_\infty^\beta$-quasicontinuous functions for which
\begin{align}\label{quasinorm}
\|f\|_{L^1(\Haus_\infty^\beta)} :=\int_{\mathbb{R}^d} |f| \;d\mathcal{H}_\infty^\beta <+\infty,
\end{align}
which is equivalent to the completion of $C_c(\mathbb{R}^d)$ in the quasi-norm \eqref{quasinorm}.  By D.R. Adams' duality result \cite[Proposition 1 on p.~118]{Adams:1988}, this inequality implies (and is equivalent to)
\begin{align*}
\|\M_{\F,d-\beta} \mu\|_{L^1(\nu)} \leq C\|\mu\|_{DS_\beta(\mathbb{R}^d)}
\end{align*}
for every Radon measure $\nu$ such that $\nu(B(x,r)) \leq C'r^\beta$ for every $B(x,r) \subset \mathbb{R}^d$.  As a consequence one has equivalent statements of Theorems \ref{dimension_theorem} and \ref{dimension_theorem-fractional-maximal}, Corollary \ref{endpoint_sawyer}, and Remark \ref{mmz} as capacitary inequalities in terms of the Hausdorff content.

Theorem \ref{maximal bound dsbeta} motivates the following more technical analogue of Theorem \ref{dimension simplified}. 
\begin{theorem}\label{dimension complicated}
Suppose $\mu\in M_b(\mathbb{R}^d)$ is such that for some collection~$\F\subset \mathbb{Z}_+\times \mathbb{Z}_+$ we have~$\mathrm{M}_{\mathcal{F},d-\beta} \mu\in L^1(\mathcal{H}_\infty^\beta)$. Then,~$\dim_{\mathcal{H}} \mu \geq \beta$.
\end{theorem}

Finally, the dimension estimate for the spaces $DS_\beta(\mathbb{R}^d)$ then follows from the definition of the spaces as limits of sums of atoms, Theorem \ref{maximal bound dsbeta} and Theorem \ref{dimension complicated}.  In particular we obtain
\begin{theorem}\label{dimensional_estimates}
Let $\beta \in (0,d]$.  For $\mu \in DS_\beta(\mathbb{R}^d)$ one has
\begin{align*}
\operatorname*{dim_\mathcal{H}} \mu \geq \beta.
\end{align*}
\end{theorem}

The plan of the paper is as follows.  In Subsection \ref{examples}, we introduce a number of examples that show the definition of $\beta$-atoms and $\beta$-atomic spaces are not vacuous.  In Section \ref{preliminaries}, we introduce some requisite preliminaries.  In Section \ref{proofs}, we prove Theorems \ref{strong-type}, \ref{dimension_theorem}, \ref{dimension_theorem-fractional-maximal} and Corollary \ref{endpoint_sawyer}.  In Section \ref{dyadic}, we prove the simplified dimension estimate asserted in Theorem \ref{dimension simplified}.  In Section \ref{fourier} we introduce a notion of distributional $(\beta,N)$-atoms which is slightly more general than the atoms defined in the introduction.  We then show that these atoms admit various estimates which prove useful in our work, and seem likely to be useful in future research on this subject.  These estimates include a bound on their fractional grand maximal function in Lemma \ref{maximal bound} which then implies Theorem \ref{maximal bound dsbeta}.  This is the most technically difficult section of the paper and relies on some subtle Fourier analysis.  Finally, in Section \ref{dimension_estimates} we prove Theorem \ref{dimension complicated} and then from this and the preceding analysis deduce Theorem \ref{dimensional_estimates}.

\subsection{Examples}\label{examples}
In this subsection, we provide a number of examples that illustrate the richness of these function spaces.

\begin{example}\label{frostman}
For any $\beta \in (0,d)$, let $A\subset [0,1/2]^d$ be a set such that $0<\mathcal{H}^\beta(A)<+\infty$, let $\mu$ be the Frostman measure supported on $A$ (see, e.g. \cite[p.~112]{Mattila}), rescaled so that $\mu(B(x,r)) \leq c_1 r^\beta$ and $\mu(A)=c_2$ for $c_1,c_2>0$ to be chosen, and denote by $\tilde{\mu}$ the translate of $\mu$ by the vector $<1/2,1/2,\ldots,1/2>$.  Then the measure
\begin{align*}
    a:= \mu - \tilde{\mu}
\end{align*}
is a $\beta$-atom supported in $Q=[0,1]^d$.  If $c_2\leq 1/2$, the conditions (1),\;(2), and (4) are immediate from the construction, while to verify (3), one computes for any $x \in \mathbb{R}^d$ the quantity
\begin{align*}
| t^{(d-\beta)/2}e^{t\Delta} a (x)| &=t^{(d-\beta)/2} \int_{\mathbb{R}^d} \frac{1}{(4\pi t)^{d/2}}\exp(-|x-y|^2/4t)\;da(y) \\
&= t^{(d-\beta)/2} \sum_{k \in \mathbb{Z}} \int_{B(x,2^{k+1}\sqrt{t}) \setminus B(x,2^{k}\sqrt{t}) } \frac{1}{(4\pi t)^{d/2}}\exp(-|x-y|^2/4t)\;da(y)\\
&\leq t^{(d-\beta)/2} \sum_{k \in \mathbb{Z}} \frac{1}{(4\pi t)^{d/2}}  \exp(-2^{2k-2}) 2c_1(2^{k+1}\sqrt{t})^\beta\\
&\leq \sum_{k \in \mathbb{Z}} \frac{1}{(4\pi)^{d/2}}  \exp(-2^{2k-2}) 2c_1(2^{k+1})^\beta\\
&\leq 1,
\end{align*}
for a sufficiently small choice of $c_1$.
\end{example}

\begin{example}\label{hardyatoms}
    For $\beta=d$ the notion of $\beta$ atom agrees with $L^\infty$ atoms for the Hardy space $\mathrm{H}^1(\mathbb{R}^d)$ up to rescaling.  In particular, if $a$ is a $d$-atom, then $a$ satisfies (1), (2), the classical support and cancellation conditions, while Fatou's lemma and ~(3) imply
\begin{align*}
\sup_{x\in \mathbb{R}^d} |a(x)| \leq  \sup_{x\in \mathbb{R}^d}  \liminf_{t\to 0}| e^{t\Delta} a (x)| \leq \frac{1}{|Q|}.
\end{align*} 
Conversely, if $a$ is an $L^\infty$ atom, then $a$ satisfies (1),\;(2),\;(4), and with a suitable rescaling (3) follows as in Example \ref{frostman}.
\end{example}

\begin{example}\label{loop atom}
Let $\Gamma \subset [0,1]^d$ be a closed piecewise $C^1$ curve with length $1$ and~$C^1$ parametrization $\gamma\colon [0,1] \to \mathbb{R}^d$. Define the associated Radon measure $\mu_\Gamma$ by
\begin{align*}
\int_{\mathbb{R}^d} \Phi \cdot d\mu_\Gamma := \int_0^{1} \Phi(\gamma(t))\cdot \dot{\gamma}(t)\;dt
\end{align*}
for $\Phi \in C_c(\mathbb{R}^d;\mathbb{R}^d)$.  If
\begin{align*}
|\mu_\Gamma|(B(x,r))\leq c_1 r
\end{align*}
for $c_1>0$ sufficiently small, then for each $i= 1,\ldots, d$, the component
\begin{align*}
a:=(\mu_{\Gamma})_i
\end{align*}
is a $1$-atom.  As before, the conditions (1) and (4) are immediate from the construction, while the verification of (3) is identical to the  computation in Example \ref{frostman}.  Finally, the closedness of $\Gamma$ implies
\begin{align*}
    \int_{\mathbb{R}^d} \Phi_i \;d(\mu_\Gamma)_i = 0
\end{align*}
for $\Phi_i \equiv 1$ on $[0,1]^d$, which is the cancellation condition (2).
\end{example}

\begin{lem}
   For $0<\alpha < \beta \leq d$, if $a$ is a $\beta$-atom, then $a$ is an $\alpha$-atom.
    \end{lem}
    \begin{proof}
    For $t \leq l(Q)^2$ by the property $(3)$ one has
    \begin{align*}
    |t^{(d-\alpha)/2} e^{t\Delta} a (x)| &=t^{(\beta-\alpha)/2} |t^{(d-\beta)/2} e^{t\Delta} a (x)|\\
    &\leq \frac{l(Q)^{\beta-\alpha}}{l(Q)^\beta},
    \end{align*}
    while if $t > l(Q)^2$ the property $(4)$ yields
    \begin{align*}
    |t^{(d-\alpha)/2} e^{t\Delta} a (x)| &\leq t^{(d-\alpha)/2} \frac{1}{(4\pi t)^{d/2}}\\
    &\leq \frac{1}{l(Q)^\alpha}.
    \end{align*}
    \end{proof}

\begin{cor}\label{nested}
    If $0<\alpha < \beta \leq d$, \begin{align*}
    DS_\beta(\mathbb{R}^d) \subset DS_\alpha(\mathbb{R}^d).
    \end{align*}
    \end{cor}
    \begin{proof}
    This follows easily from the definition of the spaces and the preceding Lemma. 
    \end{proof}  

\begin{example}
    Let $\beta=d$.  Then $DS_d(\mathbb{R}^d) \equiv \mathrm{H}^1(\mathbb{R}^d)$.  From Example \ref{hardyatoms} we know that the atoms in the two spaces are the same.  This easily gives the inclusion of the Hardy space in $DS_d(\mathbb{R}^d)$ because the Hardy space consists of distributions which are limits of a fixed series of atoms: 
    \begin{align}\label{hardy}
      f=\sum_{i=1}^\infty \lambda_i a_i.
    \end{align}
     To show the reverse inclusion, note that for any element of $DS_d(\mathbb{R}^d)$ one has the uniform bound
\begin{align*}
\left\|\sum_{i=1}^n \lambda_{i,n} a_{i,n}\right\|_{\mathrm{H}^1(\mathbb{R}^d)} \lesssim \sum_{i=1}^n |\lambda_{i,n}| <\infty.
\end{align*}
That is, the sequence is not only bounded in the space of measures but also in the Hardy space.  The claim now follows from compactness: $\mathrm{H}^1(\mathbb{R}^d)=(VMO(\mathbb{R}^d))'$ gives that up to a subsequence, one can extract a limit in the weak-star topology of $\mathrm{H}^1(\mathbb{R}^d)$.  This limit coincides with the weak-star limit in the sense of measures by density of continuous compactly supported functions in $VMO(\mathbb{R}^d)$, and therefore every element of $DS_d(\mathbb{R}^d)$ is an element of the Hardy space. 
\end{example}

The preceding example shows that the space $DS_d(\mathbb{R}^d)$ is equivalent to the Hardy space $ \mathrm{H}^1(\mathbb{R}^d)$ by establishing that our weaker notion of sequential approximation coincides with the notion in the Hardy space.  We conclude this Section with the example that motivated the definition of these atoms and atomic spaces.

\begin{example}\label{loop embedding}
Let $F \in M_b(\mathbb{R}^d;\mathbb{R}^d)$ satisfy $\operatorname*{div}F=0$ in the sense of distributions.  Then $F_i \in DS_1(\mathbb{R}^d)$ for all $i=1,\ldots,d$.  This follows from a rescaling of the conclusion of \cite[Theorem 1.5]{HS}.
\end{example}

\section{Preliminaries}\label{preliminaries}

\begin{lem}\label{StructuralLemma}
Let~$\beta < d$. Let~$U = \cup_\gamma B(x_\gamma, r_\gamma)$ be an open bounded subset of~$\R^d$. There exists a covering of~$U$ by dyadic cubes~$Q_j$ such that
\eq{\label{GoodHausdorffContent}
\sum_j \ell(Q_j)^\beta \lesssim \Haus_\infty^\beta(U)
} 
and for any~$\gamma$ there exists a cube~$Q_j$ in our covering such that it intersects~$B(x_\gamma, r_\gamma)$ and~$\ell(Q_j) \geq c r_\gamma$; here~$c>0$ is a constant that depends only on the dimension and $\beta$.
\end{lem}

\begin{proof}
By the equivalence of the dyadic Hausdorff content and Hausdorff content, we may find a countable collection of dyadic cubes such that
\begin{align*}
\sum_j \ell(Q_j)^\beta \lesssim \Haus_\infty^\beta(U).
\end{align*}
If for every $B(x_\gamma, r_\gamma)$, the condition $\ell(Q_j) \geq c r_\gamma$ is satisfied we are done.  Otherwise let $\mathcal{B}=\{ \gamma\}$ such that this condition is violated.  For each $\gamma \in \mathcal{B}$ one has $\ell(Q_j) < c r_\gamma$ for every $Q_j$ which intersects $B(x_\gamma,r_\gamma)$.  Because $U$ is bounded, $R:=\sup_{\gamma \in \mathcal{B}} r_\gamma < +\infty$, and in the first iteration we choose a $\gamma \in \mathcal{B}$ such that $r_\gamma>R/2$.  The ball $B(x_\gamma, 2r_\gamma)$ intersects $L$ dyadic cubes $Q_i'$ with $L\leq 2^d$ such that $l(Q_i')/2 <2r_\gamma \leq l(Q_i')$.  Throw away all of the cubes from the original collection which intersect $B(x_\gamma,r_\gamma)$ and add these $L$ cubes which intersect $B(x_\gamma, 2r_\gamma)$.  The collection remains a cover, $l(Q_i') \geq 2 r_\gamma$, and we claim that for these cubes
\begin{align*}
\sum_{i=1}^L \ell(Q'_i)^\beta \leq \sum_{Q_j \cap B(x_\gamma,r_\gamma)} \ell(Q_j)^\beta - c'r_\gamma^\beta
\end{align*}
for some $c'>0$.  This is true because $L\leq 2^d$, and with the bounds on the lengths one has
\begin{align*}
\sum_{i=1}^{L} l(Q_i')^\beta \leq 2^d 4^\beta r_\gamma^\beta
\end{align*}

Meanwhile, observe that because the cubes form a cover of the ball, one has
\begin{align*}
\omega_d r_\gamma^d  &\leq \sum_{Q_j \cap B(x_\gamma,r_\gamma)} l(Q_j)^d \\
&= \sum_{Q_j \cap B(x_\gamma,r_\gamma)} l(Q_j)^{d-\beta}l(Q_j)^\beta \\
&<\sum_{Q_j \cap B(x_\gamma,r_\gamma)} (cr_\gamma)^{d-\beta}l(Q_j)^\beta,
\end{align*}
and therefore
\begin{align*}
\frac{\omega_d}{c^{d-\beta}} r_\gamma^\beta \leq \sum_{Q_j \cap B(x_\gamma,r_\gamma)}l(Q_j)^\beta.
\end{align*}
Choose $c$ such that
\begin{align*}
 2^d 4^\beta  = \frac{1}{2} \frac{\omega_d}{c^{d-\beta}}
\end{align*}
Then
\begin{align*}
\sum_{i=1}^{L} l(Q_i')^\beta &\leq 2^d 4^\beta r_\gamma^\beta \\
&=\frac{\omega_d}{c^{d-\beta}} r_\gamma^\beta-  \frac{1}{2}\frac{\omega_d}{c^{d-\beta}} r_\gamma^\beta \\
&\leq \sum_{Q_j \cap B(x_\gamma,r_\gamma)}l(Q_j)^\beta - c'r_\gamma^\beta
\end{align*}
for $c'= \frac{1}{2}\frac{\omega_d}{c^{d-\beta}}$.

One then iterates this argument, with the observation that because one subtracts $c'r_\gamma^\beta$, the $r_\gamma$ selected must tend to zero, as the Hausdorff content of the set is strictly positive.  This means that one does not change large cubes, and so the cover stabilizes.

\end{proof}

\begin{cor}\label{StructuralCorollary}
Let~$\beta < d$.  Let~$U = \cup_\gamma B(x_\gamma, r_\gamma)$ be an open bounded subset of~$\R^d$. We may find a covering of~$U$ by balls~$B(\tilde{x}_j, \tilde{r}_j)$ such that~$\sum \tilde{r}_j^\beta \lesssim \Haus_\infty^\beta(U)$ and each ball~$B(x_\gamma,r_\gamma)$ is covered entirely by a ball of the covering.
\end{cor}

\begin{proof}
An application of Lemma \ref{StructuralLemma} yields dyadic cubes which satisfy \eqref{GoodHausdorffContent} and whose lengths satisfy $\ell(Q_j) \geq c r_\gamma$.  The desired cover is then obtained by dilates of balls containing these dyadic cubes.
\end{proof}

\section{Sobolev Inequalities}\label{proofs}
In this Section we establish the several Sobolev inequalities of the paper.  We begin with the proof of Theorem \ref{strong-type}.
\begin{proof}[Proof of Theorem \ref{strong-type}]
The argument follows that in the paper \cite{SpectorHernandezRaita2021}, which we recall for the convenience of the reader.  By the definition of the atomic space, it suffices to show that there is a universal constant $C=C(\alpha,\beta,d)>0$ such that
\begin{align}\label{sufficient}
\int_0^\infty t^{\alpha/2-1} \|e^{t\Delta} a\|_{L^{d/(d-\alpha),1}(\mathbb{R}^d)}\;dt \leq C
\end{align}
for every $\beta$-atom.  If one interpolates the usual convolution estimates
\begin{align}
\|e^{t\Delta} a\|_{L^{1}(\mathbb{R}^d)} \leq \|a\|_{L^1(\mathbb{R}^d)} \label{L1 convolution} \\
\|e^{t\Delta} a\|_{L^{\infty}(\mathbb{R}^d)} \leq \frac{1}{(4\pi t)^{d/2}} \|a\|_{L^1(\mathbb{R}^d)}, \label{Linfinity}
\end{align}
one finds a logarithmic divergence of the integral in \eqref{sufficient}.  However, the definition of atoms gives a slightly better estimate than the second for small values of $t$:
\begin{align}
\|e^{t\Delta} a\|_{L^{\infty}(\mathbb{R}^d)} \leq \frac{1}{t^{(d-\beta)/2} l(Q)^\beta}, \label{Linfinityprime}
\end{align}
which by interpolation with \eqref{L1 convolution} (and the fact that $\|a\|_{L^1(Q)} \leq 1$) implies that
\begin{align*}
\int_0^{l(Q)^2} t^{\alpha/2-1} \|e^{t\Delta} a\|_{L^{d/(d-\alpha),1}(\mathbb{R}^d)}\;dt \leq C_1.
\end{align*}
We will next show that
\begin{align}
\|e^{t\Delta} a\|_{L^{1}(\mathbb{R}^d)} \leq C\max\left\{\frac{l(Q)^{d-\beta}}{t^{(d-\beta)/2}},\frac{l(Q)}{t^{1/2}}\right\} \label{L1prime}
\end{align}
for $t \geq l(Q)^2$.  

Without loss of generality we may assume that the cube is centered at zero.  We then observe that for $x\in (2Q)^c,y \in Q$ we have
\begin{align*}
|\exp(-|x-y|^2/4t)-\exp(-|x|^2/4t)| \lesssim  \frac{|y|}{t^{1/2}} |\exp(-c|x|^2/t),
\end{align*}
for some $c>0$, and therefore
\begin{align*}
\int_{(2Q)^c} &|e^{t\Delta} a(x)|\;dx \\
&= \frac{1}{(4\pi t)^{d/2}} \int_{(2Q)^c} \left|\int_{Q} \left[\exp(-|x-y|^2/4t)-\exp(-|x|^2/4t)\right]a(y)\;dy \right|\;dx \\
&\lesssim \frac{l(Q)}{ t^{1/2}} \|a\|_{L^1(Q)}.
\end{align*}
A combination of this and \eqref{Linfinityprime} yields
\begin{align*}
\| e^{t\Delta}a\|_{L^{1}(\mathbb{R}^d)} &= \int_{2Q} |e^{t\Delta}  a|\;dx + \int_{(2Q)^c} |e^{t\Delta}a|\;dx \\
&\leq \frac{t^{(\beta-d)/2}}{l(Q)^\beta} (2l(Q))^d + \frac{l(Q)}{t^{1/2}}, \\
&\leq C' \left(\frac{l(Q)^{d-\beta}}{t^{(d-\beta)/2}} + \frac{l(Q)}{t^{1/2}}\right)
\end{align*}

The conclusion of the theorem then follows by the interpolation of \eqref{L1prime} and \eqref{Linfinity}, which is sufficient to obtain the bound
\begin{align*}
\int_{l(Q)^2}^\infty t^{\alpha/2-1} \|e^{t\Delta} a\|_{L^{d/(d-\alpha),1}(\mathbb{R}^d)}\;dt \leq C_2.
\end{align*}

\end{proof}

We next give the proof of Theorem \ref{dimension_theorem}.

\begin{proof}[Proof of Theorem \ref{dimension_theorem}]
As in the preceding theorem, it suffices to prove the estimate for a single atom and without loss of generality we may assume this atom is supported in a cube centered at zero. For $\alpha>d-\beta$, we have
\begin{align*}
|I_\alpha a(x)| &\leq \frac{1}{\Gamma(\alpha/2)}\int_0^\infty t^{\alpha/2-1}|e^{t\Delta} a(x)|\;dt\\
&\leq \frac{1}{\Gamma(\alpha/2)}\int_0^{r} t^{(\alpha+\beta-d)/2-1} t^{(d-\beta)/2} |e^{t\Delta} a(x)|\;dt\\
&\;\;+\frac{1}{\Gamma(\alpha/2)}\int_{r}^{\infty} t^{\alpha/2-1}  |e^{t\Delta}  a(x)|\;dt \\
&\leq C \left(\frac{r^{(\alpha+\beta-d)/2}}{l(Q)^\beta} + r^{(\alpha-d)/2} \|a\|_{L^1(Q)}\right)\\
&=C' l(Q)^{\alpha-d}
\end{align*}
(with the choice $r=l(Q)^2$) for all $x \in \mathbb{R}^d$.  We also claim the decay estimate
\begin{align*}
|I_\alpha a (x)| \leq l(Q) \frac{C}{|x|^{d-\alpha+1}}
\end{align*}
if $x \notin 2Q$.  Here we use
\begin{align*}
I_\alpha a (x) = \frac{1}{\gamma(\alpha/2)} \int_{\mathbb{R}^d} \left[\frac{1}{|x-y|^{d-\alpha}} - \frac{1}{|x|^{d-\alpha}} \right]a(y)\;dy
\end{align*}
and
\begin{align*}
\left|\frac{1}{|x-y|^{d-\alpha}} - \frac{1}{|x|^{d-\alpha}} \right| \lesssim \frac{ |y|}{|x|^{d-\alpha+1}}.
\end{align*}
Therefore as $\nu(Q)\leq l(Q)^{d-\alpha}$ we find
\begin{align*}
\int_{\mathbb{R}^d} |I_\alpha a(x)| \;d\nu &= \int_{2Q} |I_\alpha a(x)| \;d\nu + \int_{(2Q)^c} |I_\alpha a(x)| \;d\nu \\
&\leq C' l(Q)^{\alpha-d} (2l(Q))^{d-\alpha} + l(Q) \int_{(2Q)^c} \frac{1}{|x|^{d-\alpha+1}}\;d\nu
\end{align*}
The result follows from the fact that
\begin{align*}
l(Q) \int_{(2Q)^c} \frac{1}{|x|^{d-\alpha+1}}\;d\nu \leq C,
\end{align*}
which can be argued by a dyadic expansion as in \cite{RSS}.
\end{proof}

We next give the proof of Theorem \ref{dimension_theorem-fractional-maximal}.

\begin{proof}[Proof of Theorem \ref{dimension_theorem-fractional-maximal}]
As in the two preceding theorems, it suffices to show that
\begin{align*}
 \| \sup_{t>0} t^{(d-\beta)/2}|e^{t\Delta}a|\|_{L^{1}(\nu)}  \leq C
\end{align*}
for any $\beta$-atom $a$ supported in a cube centered at zero.  As in Theorem \ref{strong-type}, we split the estimate into the local and global pieces
\begin{align*}
    \left\| \sup_{t>0} t^{(d-\beta)/2} |e^{t\Delta}a| \right\|_{L^{1}(\nu)} &= \int_{2Q} \sup_{t>0} t^{(d-\beta)/2} |e^{t\Delta}  a|\;d\nu + \int_{(2Q)^c} \sup_{t>0} t^{(d-\beta)/2} |e^{t\Delta}a|\;d\nu.
\end{align*}
The assumption $a$ is a $\beta$-atom yields for the first term the bound
\begin{align*}
\int_{2Q} \sup_{t>0} t^{(d-\beta)/2} |e^{t\Delta}  a|\;d\nu &\leq \int_{2Q} \frac{1}{l(Q)^\beta} \;d\nu \\
&\leq C.
\end{align*}
For the second term, for $x\in (2Q)^c,y \in Q$ the inequality
\begin{align*}
|\exp(-|x-y|^2/4t)-\exp(-|x|^2/4t)| \lesssim  \frac{|y|}{t^{1/2}} \exp(-c|x|^2/t),
\end{align*}
for some $c<1/4$ implies
\begin{align*}
\sup_{t>0} t^{(d-\beta)/2} |e^{t\Delta}  a| &\lesssim \sup_{t>0} \frac{l(Q)}{t^{(\beta+1)/2}} \exp(-c|x|^2/t) \\
&\lesssim \frac{l(Q)}{|x|^{\beta+1}},
\end{align*}
as optimization yields $t^2\approx |x|$ in the computation of the supremum.  Therefore
\begin{align*}
\int_{(2Q)^c} \sup_{t>0} t^{(d-\beta)/2} |e^{t\Delta}  a| \;d\nu &\leq \int_{(2Q)^c} \frac{l(Q)}{|x|^{\beta+1}}d\nu\\
&\leq C
\end{align*}
by another dyadic splitting argument.  This completes the proof of the desired bound.
\end{proof}

We conclude this section with the proof of Corollary \ref{endpoint_sawyer}.

\begin{proof}[Proof of Corollary \ref{endpoint_sawyer}]
As in Example \ref{loop embedding}, the atomic decomposition from \cite{HS} shows that
\begin{align*}
\left\{ F \in M_b(\mathbb{R}^d;\mathbb{R}^d) : \operatorname*{div}F=0\right\} \hookrightarrow [DS_{1}]^d.
\end{align*}
Therefore in combination with Theorem \ref{dimension_theorem-fractional-maximal} and Corollary \ref{nested} one finds, for any of the components $F_i$ of $F$ and any $\beta \geq 1$
\begin{align*}
     \| \sup_{t>0} t^{(d-\beta)/2} |e^{t\Delta}F_i|\|_{L^{1}(\nu)} &\leq C\|F_i\|_{DS_{\beta}(\mathbb{R}^d)}\\
     &\leq C \|F_i\|_{DS_{1}(\mathbb{R}^d)}\\
     &\leq C\|F\|_{L^1}.
\end{align*}
This completes the proof.
\end{proof}



\section{A Simplified Dimension Estimate}\label{dyadic}
For $\beta \in [0,d)$ we introduce a truncated analogue of the dyadic fractional maximal function which only computes the supremum over cubes with length larger than some parameter $l$:
\begin{align}\label{truncated}
\mathcal{M}^{Q_0}_{d-\beta,l}(\mu)(x):=\sup_{\substack{Q \in \mathcal{D}(Q_0),\\l(Q)\geq l}} \chi_Q(x) \frac{|\mu(Q)|}{\;\;l(Q)^{\beta}},
\end{align}

\begin{proof}[Proof of Theorem \ref{dimension simplified}]
As dimension is a local property, it suffices to prove that for $A \subset \tau +Q_0$, $\mathcal{H}^\beta(A)=0$ implies $|\mu|(A)=0$. Without loss of generality restrict our consideration to $A \subset Q_0$.  Here we will actually prove a stronger quantitative version of the dimensional estimate: For all~$\epsilon > 0 $ there exists~$\delta > 0$ such that for any Borel set~$A \subset Q_0$ with~$\Haus^{\beta,Q_0}_\infty(A) < \delta$ we also have~$|\mu|(A) < \epsilon$.  Here we utilize the dyadic Hausdorff content $\Haus^{\beta,Q_0}_\infty$, whose definition is given by
\begin{align*}
    \Haus^{\beta,Q_0}_\infty(A) := \inf \left\{ \sum_{i=1}^\infty l(Q_i)^\beta : A \subset \cup_{i=1}^\infty Q_i, Q_i \in \mathcal{D}(Q_0)\right\}.
\end{align*}

Let $\epsilon>0$ be given. First, we claim that we can use absolute continuity of the integral to choose $\delta>0$ sufficiently small such that if $\Haus^{\beta,Q_0}_\infty(A) < \delta$ then
\begin{align}\label{absolute continuity}
\int_A \mathcal{M}^{Q_0}_{d-\beta}(\mu) \;d\mathcal{H}^{\beta,Q_0}_\infty <\frac{\epsilon}{32}.
\end{align}
Indeed, from the assumption \eqref{L1}, we can find $M>0$ such that
\begin{align*}
\int_M^\infty \mathcal{H}_\infty^{\beta,Q_0}(\{x \in Q_0: \mathcal{M}^{Q_0}_{d-\beta}(\mu)(x)>t\})\;dt< \epsilon/64.
\end{align*}
For this choice of $M$, for any $\delta \leq \epsilon/64M$ one has
\begin{align*}
\int_A \mathcal{M}^{Q_0}_{d-\beta}(\mu) \;d\mathcal{H}^{\beta,Q_0}_\infty &\leq \int_0^M \mathcal{H}_\infty^{\beta,Q_0}(\{x \in A : \mathcal{M}^{Q_0}_{d-\beta}(\mu)(x)>t\})\;dt\\
\;\;&+ \int_M^\infty \mathcal{H}_\infty^{\beta,Q_0}(\{\mathcal{M}^{Q_0}_{d-\beta}(\mu)(x)>t\})\;dt\\
&\leq \epsilon/64+ \epsilon/64 = \epsilon/32.
\end{align*}

Second, we claim we can restrict our attention to
\begin{align*}
A':=\{ x \in A : l_x \geq \frac{1}{N_\epsilon}\}
\end{align*}
for some sufficiently large $N_\epsilon$, where $l_x$ is such that
\begin{align*}
|\mu|(Q) \leq 2 |\mu(Q)|
\end{align*}
for every $x \in l(Q)\leq l_x$.  Indeed, by the Lebesgue differentiation theorem, we have that
\begin{align*}
\lim_{Q\to x} \frac{|\mu(Q)|}{|\mu|(Q)} = 1
\end{align*}
for $\mu$ almost every $x \in \mathbb{R}^d$.  As a consequence, for every $x \in A$, there exists $l_x>0$ such that
\begin{align*}
|\mu|(Q) \leq 2 |\mu(Q)|
\end{align*}
for $l(Q)\leq l_x$.  Using $|\mu|(Q_0)<+\infty$, this implies 
\begin{align*}
\lim_{n \to \infty} |\mu|( \{ x \in Q_0 : 0<l_x<\frac{1}{n}\}) =|\mu|( \{ x \in Q_0 : l_x=0\}) =0
\end{align*}
and therefore we can find $N_\epsilon \in \mathbb{N}$ such that
\begin{align}
|\mu|( \{ x \in Q_0 : 0<l_x<\frac{1}{n}\}) <\frac{\epsilon}{32} \label{small_cube_control}
\end{align}
for every $n\geq N_\epsilon$.
Thus, we can estimate
\begin{align}
|\mu|(A) &\leq |\mu|( A') + \frac{\epsilon}{32}.\label{large_cubes_remain}
\end{align}

Third, we replace $A'$ with a dyadic covering subordinate to $Q_0$ by using the definition of the dyadic Hausdoroff content:  For any $\delta>0$, by the definition of $\Haus^{\beta,Q_0}_\infty(A)$ there exists a family of dyadic cubes subordinate to $Q_0$ such that
\begin{align*}
A \subset \bigcup_{i=1}^\infty& Q_i\\
\sum_{i=1}^\infty l(Q_i)^\beta &< \delta\\
Q_i \cap Q_j &= \emptyset \text{ if } i \neq j.
\end{align*}
In combination with our first restriction on $\delta$, for any $\delta \leq \min\{\epsilon/64M,\frac{1}{2N_\epsilon^\beta}\}$ we can ensure the cubes are small enough that can utilize the consequences of the Lebesgue differentiation theorem for the remaining sets to estimate, and in particular we only need to consider cubes for which
\begin{align*}
A' \cap Q_i &\neq \emptyset.
\end{align*}
Define
\begin{align*}
A_n&:= \cup_{i=1}^n Q_i \\
A_\infty&:= \cup_{i=1}^\infty Q_i.
\end{align*}
Then
\begin{align*}
|\mu|(A') &= |\mu|(A'\setminus A_n) + |\mu|(A'\cap A_n) \\
&\leq |\mu|(A_\infty \setminus A_n)+ |\mu|(A'\cap A_n) \\
&\leq \frac{\epsilon}{32} + |\mu|(A_n)
\end{align*}
for $n \geq M_\epsilon \in \mathbb{N}$. As
\begin{align*}
\Haus^{\beta,Q_0}_\infty(A_n)\leq \sum_{i=1}^n l(Q_i)^\beta <\delta,
\end{align*}
the absolute continuity of the integral noted above implies
\begin{align*}
    \int_{A_n} \mathcal{M}^{Q_0}_{d-\beta}(\mu) \;d\mathcal{H}^{\beta,Q_0}_\infty \leq \frac{\epsilon}{32}.
\end{align*}
If we define
\begin{align*}
    l_n := \min_{i=1,\ldots,n} l(Q_i)
\end{align*}
this implies
\begin{align*}
    \int_{A_n} \mathcal{M}^{Q_0}_{d-\beta,l_n}(\mu) \;d\mathcal{H}^{\beta,Q_0}_\infty \leq \frac{\epsilon}{32},
\end{align*}
where $\mathcal{M}^{Q_0}_{d-\beta,l_n}(\mu)$ is the truncated maximal function introduced in \eqref{truncated}.

Finally, if we can show that
\begin{align*}
|\mu|(A_n) \leq \frac{\epsilon}{2}
\end{align*}
then
\begin{align*}
    |\mu|(A) &\leq |\mu|(A') + \frac{\epsilon}{32}\\
    &\leq |\mu|(A_n) + \frac{\epsilon}{16} \\
    &\leq \frac{\epsilon}{2} +\frac{\epsilon}{16} \\
    &<\epsilon.
\end{align*}
which is the desired conclusion.  To this end, for each $k \in \mathbb{Z}$, define
\begin{align*}
\Omega_{n,k}:= \{ x \in A_n :  2^{k+1} \geq \mathcal{M}^{Q_0}_{d-\beta,l_n}(\mu)(x) >2^k\}
\end{align*}
By the definition of the Hausdorff content, we can now find a family of cubes $\{Q_{n,k,j}\}_{j \in \mathbb{N}}$ such that
\begin{align*}
\Omega_{n,k} \subset \bigcup_{j} Q_{n,k,j}
\end{align*}
and
\begin{align*}
  \sum_{j } l(Q_{n,k,j})^\beta    \leq 2 \mathcal{H}^{\beta,Q_0}_\infty(\Omega_{n,k}).
\end{align*}
As
\begin{align*}
\mathcal{H}^{\beta,Q_0}_\infty(\Omega_{n,k}) \leq \mathcal{H}^{\beta,Q_0}_\infty(A_n) <\delta,
\end{align*}
we note that from the choice of $\delta \leq \frac{1}{2N_\epsilon^\beta}$ we have that $l(Q_{n,k,j})\leq \frac{1}{N_\epsilon}$.

Then the truncation on the maximal function implies it is nowhere infinite (and in fact has a uniform upper bound which depends on $n$ through $l_n$) allows us to write
\begin{align*}
    A_n = \bigcup_{k \in \mathbb{Z}} \Omega_{n,k}
\end{align*}
and therefore
\begin{align*}
     A_n = \bigcup_{k \in \mathbb{Z}} \Omega_{n,k} \subset \bigcup_{k \in \mathbb{Z}} \bigcup_{j \in \mathbb{N}} Q_{n,k,j} 
\end{align*}

Putting these several facts together, and discarding any cubes which do not intersect $A_n$, we obtain
\begin{align*}
|\mu|(A_n)&\leq \sum_{k\in \mathbb{Z},j\in \mathbb{N}} |\mu|(Q_{n,k,j}) \\
&= \sum_{k\in \mathbb{Z},j\in \mathbb{N}} l(Q_{n,k,j})^{\beta}  l(Q_{n,k,j})^{-\beta} |\mu|(Q_{n,k,j}) \\
&\leq 2\sum_{k\in \mathbb{Z},j\in \mathbb{N}} l(Q_{n,k,j})^{\beta}  l(Q_{n,k,j})^{-\beta} |\mu(Q_{n,k,j})| \\
&\leq 2 \sum_{k\in \mathbb{Z},j\in \mathbb{N}} l(Q_{n,k,j})^{\beta}  2^{k+1} \\
&\leq 8 \sum_{k\in \mathbb{Z}} 2^{k}\mathcal{H}^{\beta,Q_0}_\infty(\Omega_{n,k})\\
&\leq 16 \int_{A_n} \mathcal{M}^{Q_0}_{d-\beta,l_n}(\mu)d\mathcal{H}^{\beta,Q_0}_\infty\\
&\leq \frac{16 \epsilon}{32} \\
&= \frac{\epsilon}{2},
\end{align*}
which completes the demonstration of the claim and therefore the proof.

\end{proof}

\section{Maximal Function Estimates via Littlewood-Paley theory}\label{fourier}
In the sequel we will utilize the following lemma about the existence of a compactly supported function whose Fourier transform does not vanish on a ball.
\begin{lem}\label{lemma1}
There exists a function $\Psi \in C^\infty_c(\mathbb{R}^d)$ such that $\widehat{\Psi}(\xi) \geq 1$ on $B(0,1)$. 
\end{lem}
\begin{proof}
The standard mollifier defined by 
\begin{align*}
    \rho(x) = \begin{cases} 
      c\exp\left(\frac{-1\phantom{xx}}{1-|x|^2}\right) & |x|\leq 1 \\
      0 & |x|>1 
   \end{cases}
\end{align*}
is a non-negative, radial, $C^\infty_c(\mathbb{R}^d)$ function which satisfies
\begin{align*}
  \widehat{\rho}(0) &=  \int_{\mathbb{R}^d} \rho(x)\;dx = 1\\
  |\nabla \widehat{\rho}(\xi)| &=  \left|\int_{\mathbb{R}^d} 2\pi i x e^{2\pi i x\cdot \xi} \rho(x)\;dx \right| \leq 2\pi.
\end{align*}
In particular, $\widehat{\rho}$ is Lipschitz with constant at most $2\pi$, so that
\begin{align*}
  \widehat{\rho}(\xi) \geq 1/2
\end{align*}
for $\xi \in B(0,1/4\pi)$.  The desired function can then be obtained by suitable rescaling:
\begin{align*}
    \widehat{\Psi}(\xi) = 2\widehat{\rho}(\xi/4\pi),
\end{align*}
i.e.
\begin{align*}
    \Psi(x) = 2 (4\pi)^d \rho(4\pi x).
\end{align*}
\end{proof}

\begin{definition}
A distribution~$a\in \mathcal{S}'(\mathbb{R}^d)$ is called a distributional~$(\beta,N)$-atom provided it is supported in a cube~$Q$ with sidelength~$2\ell(Q)$, 
\begin{align*}
< x^\alpha a, \varphi> = 0
\end{align*}
for any homogeneous polynomial $x^\alpha$ with $|\alpha| \leq N$ and $\varphi \equiv 1$ in a neighborhood of $Q$, and satisfies the uniform bound
\begin{equation}\label{HeatExtensionBound}
    \|e^{t\Delta} a\|_{L^\infty(\mathbb{R}^d)} \leq (\ell(Q))^{-\beta}t^{\frac{\beta - d}{2}},\qquad t > 0.
\end{equation}
\end{definition}
As in the introduction, $e^{t\Delta}$ denotes the heat extension, whose action on a summable function $f$ is given by the convolution
\eq{
e^{t\Delta}f(x) = (4\pi t)^{-\frac{d}{2}}\int\limits_{\mathbb{R}^d}f(x-y)e^{-\frac{|y|^2}{4t}}\,dy,\qquad x\in \R^d,
} 
and extended by continuity to distributions. The role of the number~$N$ is minor: if we say~$a$ is a~$\beta$-atom, we mean it is a~$(\beta,N)$ atom for some~$N$; usually, this number is assumed to be sufficiently large (though sometimes~$N=0$ is also interesting). We say that the atom~$a$ is \emph{adapted} to~$Q$ if~$Q$ and~$a$ suit the definition above. 

\begin{rem}\label{DilationRemark}
Let~$a$ be adapted to a cube~$Q$ with sidelength~$2\ell(Q)$ and center~$c_Q$. Then, the dilation~$\tilde{a}$ of the distribution~$a$ given by
\eq{
\tilde{a}(x) = \ell(Q)^{d} a\big(\ell(Q) (x-c_Q)\big),\qquad x\in \R^d,
}
is a distributional~$\beta$-atom adapted to the cube~$[-1,1]^d$, i.e. it satisfies the bound
\begin{equation}\label{HeatExtensionBoundStandard}
    \|e^{t\Delta} \tilde{a}\|_{L^\infty(\mathbb{R}^d)} \leq t^{\frac{\beta - d}{2}},\qquad t > 0.
\end{equation}

\end{rem} 
We have been a little bit informal when writing the argument of a distribution. The formal definition of~$\tilde{a}$ is given by duality:
\eq{
\scalprod{\tilde{a}}{\varphi} = \scalprod{a}{\tilde{\varphi}}, \qquad \tilde{\varphi}(x) = \varphi\Big(\frac{x}{\ell(Q)} + c_Q\Big),\quad \varphi \in \mathcal{S}(\R^d),\ x\in \R^d.
}
In other words, we make a dilation of~$a$ that preserves its total mass if it is finite.

Let~$\Xi\colon \R^d \to \R$ be a smooth function such that~$\hat{\Xi}$ is compactly supported, non-negative, and is equal to~$1$ in a neighborhood of the origin. We set
\eq{
\Xi_k(x) = 2^{kd}\Xi(2^k x), \quad k\in\Z, \ x\in \R^d,
}
and define the `smoothed Littlewood--Paley projectors' via
\eq{\label{Telescopic}
P_{\leq k}[f] = f*\Xi_k;\qquad P_k[f] = P_{\leq k}[f] - P_{\leq k-1}[f].
}
Here~$f$ is an arbitrary tempered distribution and~$k\in\Z$.

Let~$a$ be a distributional~$\beta$-atom  adapted to~$[-1,1]^d$. We split it into the sum
\eq{
a = P_{\leq 0}[a] + \sum\limits_{k > 0}P_k[a]
}
and formulate the individual estimates for the summands. Before that, we note an immediate consequence of~\eqref{HeatExtensionBoundStandard}:
\eq{\label{TrivialBound}
\|P_{\leq k}[a]\|_{L^\infty(\mathbb{R}^d)}\lesssim 2^{(d-\beta)k}.
}
We will also use modified `projectors'~$\tilde{P}_k$. Let~$\tilde{\Xi}$ be a smooth function whose Fourier transform is compactly supported in~$\mathbb{R}^d\setminus \{0\}$ and equals~$1$ on the set~$\set{\xi \in \R^d}{\hat{\Xi}(\xi) \notin \{0,1\}}$.  Define
\eq{
\tilde{P}_k[f] = f*\tilde{\Xi}_k; \qquad \tilde{\Xi}_k(\xi) = 2^{dk}\tilde{\Xi}(2^k\xi), \quad k\in\Z.
}
These operators have similar properties to that of~$P_k$ and they also satisfy
\eq{\label{Composition}
\tilde{P}_k\circ P_k = P_k.
}
\begin{lem}\label{SecondLemma}
Let~$a$ be a distributional~$\beta$-atom adapted to~$Q = [-1,1]$. For any~$M \in \mathbb{N}$, the inequalities
\alg{
\label{FirstOfTwoInequalities}|P_{\leq 0}[a](x)|&\lesssim (1+|x|)^{-M};\\
\label{SecondOfTwoInequalities}|P_{k}[a](x)|&\lesssim 2^{(d-\beta)k}(1+|x|)^{-M}
}
hold true with the multiplicative constants being independent of~$k \geq 0$,~$x\in \R^d$, and the choice of~$a$.
\end{lem}
\begin{proof}
Let now~$\Psi$ be a compactly supported function such that~$\hat{\Psi}$ attains positive values on the support of~$\hat{\Xi}$. Then, in particular,
\begin{equation}\label{eq110'}
\begin{aligned}
\|\tilde{P}_k[\Psi]\|_{L^1(\mathbb{R}^d)}&\leq\|(1+|x|^2)^d\tilde{P}_k[\Psi](x)\|_{L^\infty(\mathbb{R}^d)} \\
& \lesssim \Big\|(1-\Delta)^{d}\big[\hat{\tilde{\Xi}}(2^{-k}\xi)\hat{\Psi}(\xi)\big]\Big\|_{L^1(\mathbb{R}^d)}\lesssim 2^{-Lk},\qquad k \geq 0,
\end{aligned}
\end{equation}
for any natural number~$L$. Consequently,
\mlt{\label{eq110}
\|a*\Psi\|_{L^\infty(\mathbb{R}^d)}\leq \big\|P_{\leq 0}\big[a*\Psi\big]\big\|_{L^\infty(\mathbb{R}^d)}+\sum\limits_{k> 0}\big\|P_{ k}\big[a*\Psi\big]\big\|_{L^\infty(\mathbb{R}^d)}\\
\Eeqref{Composition} \big\|P_{\leq 0}\big[a*\Psi\big]\big\|_{L^\infty(\mathbb{R}^d)}+\sum\limits_{k> 0}\big\|P_{ k}\big[a\big]*\tilde{P}_k\big[\Psi\big]\big\|_{L^\infty(\mathbb{R}^d)} \LseqrefTwo{TrivialBound}{eq110'} \sum\limits_{k\geq 0}2^{-Lk}2^{(d-\beta)k}
}
for any~$L \in \N$. The choice of~$L = d+1$ makes this series converge. We also note that~$a*\Psi$ has bounded support. Since
\eq{
P_{\leq 0}[a] = \chi*(a*\Psi),
}
where~$\hat{\chi}$ is a smooth compactly supported function, the bound~\eqref{eq110} together with the information about the support of~$a*\Psi$ leads to~\eqref{FirstOfTwoInequalities}.

The proof of~\eqref{SecondOfTwoInequalities} follows the same lines: instead of~$a*\Psi$, consider~$a*\Psi_k$ (here~$\Psi_k(x) = 2^{dk}\Psi(2^kx)$) and perform the same steps: 
\eq{\label{eq112}
\|a*\Psi_k\|_{L^\infty(\mathbb{R}^d)}\leq \|P_{\leq k}\big[a*\Psi_k\big]\|_{L^\infty(\mathbb{R}^d)}+\sum\limits_{l> k}\|P_{l}\big[a*\Psi_k\big]\|_{L^\infty(\mathbb{R}^d)} \lesssim  \sum\limits_{l\geq k}2^{-L(l-k)}2^{(d-\beta)l},
}
in the last inequality we have used
\eq{
\|\tilde{P}_l[\Psi_k]\|_{L^1(\mathbb{R}^d)} \lesssim 2^{-L(l-k)},\qquad l \geq k;
}
this follows from~\eqref{eq110'} by rescaling. Then we use
\eq{
P_{\leq k}[a] = \chi_k*(a*\Psi_k),
}
where~$\chi_k(x) = 2^{dk}\chi(2^k x)$, to derive
\eq{
|P_{\leq k}[a](x)|\lesssim 2^{(d-\beta)k}(1+|x|)^{-M}
}
from~\eqref{eq112}. The latter inequality clearly implies~\eqref{SecondOfTwoInequalities} via~\eqref{Telescopic}.
\end{proof}

Let~$\Phi$ be a Schwartz function and let ~$\Phi_t(x) = t^{-d}\Phi(x/t)$ denote the $L^1$ dilations.

\begin{st}\label{GrandMaximalAtom}
   Let~$a\in \mathcal{S}'(\R^d)$ be a distributional ~$\beta$-atom adapted to a cube~$Q$.  Then, for any~$M>d$ there exists a finite collection~$\mathcal{F}$ of Schwartz seminorms such that if~$\|\Phi\|_{\mathcal{F}} \leq 1$, then
    \eq{
        \label{GMB1}t^{d-\beta}|a*\Phi_t(x)|\lesssim (\ell(Q))^{-\beta}\Big(1+ \frac{\dist(x,Q)}{\ell(Q)}\Big)^{-M},\quad t \leq \ell(Q);
    }
    for all~$x\in \R^d$;   the bound is independent of the particular choice of~$a$.
\end{st}

The requirement on Schwartz seminorms of~$\Phi$ means there exists a large natural number~$\nu$ such that
\eq{
\Big|\frac{\partial^\gamma \Phi}{\partial x^\gamma}(x)\Big| \leq (1+|x|)^{-\nu}
}
for all~$x\in \R^d$ and all multi-indices~$\gamma$ with~$|\gamma|\leq \nu$.

\begin{rem}
As we will see from the proof, the moment conditions on~$a$ are unnecessary in Proposition~\ref{GrandMaximalAtom}.
\end{rem}

The proof of Proposition~\ref{GrandMaximalAtom} rests on Lemma~\ref{SecondLemma}. We will also need another lemma that expresses an almost-orthogonality principle.
\begin{lem}\label{AO}
Let~$k\in\mathbb{N}\cup \{0\}$ and~$t \in \R_+$ be such that~$2^k t \geq 1$. Then, for any natural numbers~$M'$ and~$M''$,
\eq{\label{AlmostOrthogonality}
|\tilde{P}_k[\Phi_t](x)| \lesssim (2^k t)^{-M''}t^{-d}(1+|x|/t)^{-M'},\qquad x\in \R^d.
}
\end{lem}
\begin{proof}
By scaling invariance, it suffices to prove the inequality
\eq{
\Big|\tilde{\Xi}_s*\Phi(y)\Big|\lesssim s^{M''}(1+|y|^2)^{-M'},\qquad s < 1,
} 
and~$\tilde{\Xi}_s(y) = s^{-d}\tilde{\Xi}(s^{-1}y)$. By the ~$L^1(\mathbb{R}^d) \to L^\infty(\mathbb{R}^d)$ boundedness of the Fourier transform, it suffices to prove
\eq{
\Big\|(1-\Delta)^{M'}\big[\hat{\tilde{\Xi}}(s\xi)\hat{\Phi}(\xi)\big]\Big\|_{L^1(\mathbb{R}^d)} \lesssim s^{M''},\qquad s < 1,
}
which follows from the requirement that~$\Phi$ is a Schwartz function and the conditions on the support of~$\hat{\tilde{\Xi}}$ (we have used a similar inequality in~\eqref{eq110'}).
\end{proof}
\begin{proof}[Proof of Proposition~\ref{GrandMaximalAtom}.]
First of all, we may dilate and shift~$a$ if needed and assume~$Q = [-1,1]^d$, see Remark~\ref{DilationRemark}. Then~\eqref{GMB1} reduces to
\eq{\label{GrandMaximalBound2}
        t^{d-\beta}|a*\Phi_t(x)|\lesssim (1+ |x|)^{-M},\quad t \leq 1.
    }
We next use the assumption $M > d$.

For the smoothed Littlewood--Paley projectors introduced above one has
\eq{\label{LPSplitting}
 t^{d-\beta}|a*\Phi_t(x)| \leq t^{d-\beta}|P_{\leq 0}[a]*\Phi_t(x)|+ \sum\limits_{k =1}^\infty t^{d-\beta}|P_{k}[a]*\Phi_t(x)|.
}
Let us consider the summands with~$2^kt \leq 1$ first (low frequency summands). By~\eqref{SecondOfTwoInequalities}, 
\eq{\label{LPPartEst}
|P_k[a](x)|\lesssim 2^{(d-\beta)k}(1+|x|)^{-M}.
}
Since~$t \leq 1$,
\eq{\label{eq29}
|\Phi_t|*(1+|x|)^{-M}\lesssim (1+|x|)^{-M},
}
and this bound depends on a finite number of Schwartz seminorms of~$\Phi$ in the sense that it follows from the inequality
\eq{
|\Phi(x)| \leq (1+|x|)^{-M},\qquad x\in \R^d
} 
since we have assumed~$M > d$. Using~\eqref{LPPartEst} and~\eqref{eq29}, 
we get
\eq{
\sum\limits_{k\colon 2^{k} t \leq 1} t^{d-\beta}|P_{k}[a]*\Phi_t(x)|\LseqrefTwo{SecondOfTwoInequalities}{eq29} t^{d-\beta}\sum\limits_{k\colon 2^{k} t \leq 1} 2^{(d-\beta)k}(1+|x|)^{-M}\lesssim (1+|x|)^{-M}.
}
The bound for the term~$t^{d-\beta}|P_{\leq 0}[a]*\Phi_t(x)|$ is completely similar (we use~\eqref{FirstOfTwoInequalities} instead of~\eqref{SecondOfTwoInequalities}). Thus,~\eqref{GrandMaximalBound2} holds true for the low frequency part of the sum~\eqref{LPSplitting}.

Note that since
\eq{\label{eq432}
t^{-d}(1+|x|/t)^{-M}*(1+|x|)^{-M} \lesssim (1+|x|)^{-M},\qquad M> d,\ t \leq 1,
}
Lemma~\ref{AO} with~$M'=M'' :=M$ yields the desired bound of the high frequency part:
\mlt{
\sum\limits_{k\colon 2^{k} t > 1} t^{d-\beta}|P_{k}[a]*\Phi_t(x)|\Eeqref{Composition}\sum\limits_{k\colon 2^{k} t > 1} t^{d-\beta}|P_{k}[a]*\tilde{P}_k[\Phi_t](x)|\\ \LseqrefThree{SecondOfTwoInequalities}{AlmostOrthogonality}{eq432} t^{d-\beta}\sum\limits_{k\colon 2^{k} t > 1}2^{(d-\beta)k} (2^k t)^{-M}(1+|x|)^{-M} \lesssim (1+|x|)^{-M},
}
since~$M >d > d-\beta$. Thus,~\eqref{GrandMaximalBound2} holds true for the high frequency part of the sum as well.

\end{proof}

\begin{st}\label{GrandMaximalAtom2}
    Let~$a\in \mathcal{S}'(\R^d)$ be a distributional~$(\beta,N)$-atom adapted to a cube~$Q$.  Then, for any~$s > 0$ there exists a finite collection~$\mathcal{F}$ of Schwartz seminorms such that if~$\|\Phi\|_{\mathcal{F}} \leq 1$, then
    \eq{
        \label{GMB2}t^{d-s}|a*\Phi_t(x)|\lesssim (\ell(Q))^{-s} \Big(1+ \frac{\dist(x,Q)}{\ell(Q)}\Big)^{-N-1-s},\quad t \geq \ell(Q),
    }
    for all~$x\in \R^d$;   the bound is independent of the particular choice of~$a$.
\end{st}

\begin{proof}
By Remark~\ref{DilationRemark}, we may assume~$Q = [-1,1]^d$. We wish to prove
\eq{\label{eq132}
t^{d-s}|a*\Phi_t(x)|\lesssim (1+ |x|)^{-N-1-s},\qquad x\in \R^d,\ t > 1.
}
Let~$M$ be a large number. Before passing to the proof, we modify~\eqref{AlmostOrthogonality} in the case~$t > 1$ a little bit:
\mlt{
(2^kt)^{-M''}t^{-d}(1+|x|/t)^{-M'} \Lref{\scriptscriptstyle 2^k \geq 1, t > 1} (2^kt)^{-M''+M'}t^{-M'}(1+|x|/t)^{-M'}\\ = (2^kt)^{-M''+M'}(t+|x|)^{-M'} \leq (2^kt)^{-M''+M'}(1+|x|)^{-M'}.
}
Thus, if we plug~$M''= 2M$ and~$M'= M$ into the latter estimate, we get from~\eqref{AlmostOrthogonality} that
\eq{\label{SecondAlmostOrthogonality}
|\tilde{P}_k[\Phi_t](x)| \lesssim(2^kt)^{-M}(1+|x|)^{-M},\qquad t > 1,\ k\in\N\cup\{0\}.
}

Let~$f = P_{\leq 0}[a]$. We may replace~$a$ with~$f$ in~\eqref{eq132} and prove
\eq{\label{EstimateForf}
t^{d-s}|f*\Phi_t(x)|\lesssim (1+ |x|)^{-N-1-s},\qquad |x|\in \R^d,\ t > 1,
}
since the difference between~$a$ and~$f$, i.e.,~$P_{> 0} [a]$, admits an estimate similar to the estimate for the case~$t \leq  1, 2^k t > 1$ in Proposition~\ref{GrandMaximalAtom}: 
\mlt{
t^{d-s}|(a-f)*\Phi_t(x)| \leq t^{d-s}\sum\limits_{k> 0}|P_k[a]*\Phi_t(x)|\\ \LseqrefTwo{FirstOfTwoInequalities}{SecondAlmostOrthogonality} \sum\limits_{k> 0}t^{d-s} 2^{(d-\beta)k}(2^kt)^{-M}(1+|x|)^{-M} \lesssim t^{d - s- M}(1+|x|)^{-M}
}
for any~$M \in \N$.

Now we concentrate on~\eqref{EstimateForf}. Note that~$f*\Phi_t$ is a convolution of two summable functions and may be treated as an integral. We split this integral into two integrals:
\eq{\label{Splitting}
f*\Phi_t(x) = t^{-d}\int\limits_{2|y| \geq |x|}f(y)\Phi\Big(\frac{x-y}{t}\Big)\,dy + t^{-d}\int\limits_{2|y| \leq |x|}f(y)\Phi\Big(\frac{x-y}{t}\Big)\,dy,
}
and estimate them individually. 

The estimate for the first part is straightforward (we use Lemma~\ref{SecondLemma}):
\eq{
t^{-d}\Big|\int\limits_{2|y| \geq |x|}f(y)\Phi\Big(\frac{x-y}{t}\Big)\,dy\Big| \lesssim
t^{-d}\int\limits_{2|y| \geq |x|} (1+|y|)^{-M}\,dy \lesssim t^{-d}(1+|x|)^{-M+d},
}
which, after multiplication by~$t^{d-s}$, does not exceed the right hand side of~\eqref{EstimateForf}, provided~$M$ is sufficiently large (recall~$s \geq 0$).

The bound for the second integral in~\eqref{Splitting} is a little bit more involved. We will use the Taylor formula
\eq{\label{Taylors}
\Phi\Big(\frac{x-y}{t}\Big) = \sum\limits_{|\gamma|\leq N}\frac{1}{\gamma!}\frac{\partial^\gamma\Phi}{\partial z^\gamma}\Big(\frac{x}{t}\Big)\Big(\frac{-y}{t}\Big)^\gamma + O\Big(\Big(\frac{|y|}{t}\Big)^{N+1}\Big(1+\frac{|x|}{t}\Big)^{-\eta}\Big),
}
which follows from Lagrange's representation of the remainder, the inequality~$|y|\leq |x|/2$, and the boundedness of a Schwartz seminorm of~$\Phi$. Here~$\eta$ is an arbitrary non-negative number (we will specify it later). 

First, we deal with the integrals that arise from plugging the monomial terms from~\eqref{Taylors} into the second integral in~\eqref{Splitting}. Note that~$f$ satisfies the cancellation condition
\eq{
\int\limits_{\R^d}f(y)y^\gamma\,dy = 0, \qquad |\gamma| \leq N,
}
since~$a$ does,~$f = P_{\leq 0}[a]$, and the symbol of the Fourier multiplier~$P_{\leq 0}$, that is,~$\hat{\Xi}$, equals one in a neighborhood of the origin. Therefore, by Lemma~\ref{SecondLemma},
\eq{
t^{-d}\Big|\int\limits_{2|y| \leq |x|}f(y)(y/t)^\gamma\,dy\Big| =
t^{-d}\Big|\int\limits_{2|y| \geq |x|}f(y)(y/t)^\gamma\,dy\Big|\lesssim t^{-d-\gamma}(1+|x|)^{-M + d + \gamma}.
}
Thus, it suffices to take~$M$ sufficiently large.

It remains to estimate the contribution of the remainder term in~\eqref{Taylors} to the second integral in~\eqref{Splitting}. Here we cannot use cancellations and perform the estimates in a more straightforward way (relying on Lemma~\ref{SecondLemma} again):
\mlt{
t^{-d} \int\limits_{2|y| \leq |x|}|f(y)|\Big(\frac{|y|}{t}\Big)^{N+1}\Big(1+\frac{|x|}{t}\Big)^{-\eta}\,dy \lesssim\\ t^{-(d+N+1)}\Big(1+\frac{|x|}{t}\Big)^{-\eta}\int\limits_{2|y| \leq |x|} (1+|y|)^{-M} |y|^{N+1}\,dy\\ \lesssim t^{-d-N-1}\Big(1+ \frac{|x|}{t}\Big)^{-\eta}; \quad M > N+1+d.
}
Thus, it suffices to verify the inequality
\eq{\label{NumericalInequality}
t^{-N-1}(1+|x|/t)^{-\eta} \lesssim t^{s}(1+|x|)^{-N-1-s}.
}
We use that~$t > 1$ and estimate the left hand side:
\eq{
t^{-N-1}(1+|x|/t)^{-\eta} \leq t^{-N-1}(1/t+|x|/t)^{-\eta} = t^{-N-1+\eta}(1+|x|)^{-\eta}.
}
We may set~$\eta = N+1+s$ and finish the proof of~\eqref{EstimateForf}.
\end{proof}

\begin{cor}\label{LemmaAboutAtomsAndRiesz}
Let~$a\in \mathcal{S}'(\mathbb{R}^d)$ be a~$(\beta, N)$-atom adapted to a cube~$Q$. Then, for any~$\alpha \in (d-\beta,d)$, we have 
\begin{equation}
    |\mathrm{I}_\alpha[a](x)|\lesssim \ell(Q)^{\alpha - d}\Big(1+ \frac{|x-c_Q|}{\ell(Q)}\Big)^{-(d-\alpha + M)},
\end{equation}
provided~$N$ is sufficiently large (depending on~$M$); here~$c_Q$ stands for the center of~$Q$ and~$\ell(Q)$ denotes the sidelength of this cube.
\end{cor}
\begin{proof}
We may represent
\eq{
\mathrm{I}_\alpha [a] = \sum\limits_{k\in \Z}2^{-\alpha k}\Phi_{2^{-k}}*a
}
for a specific Schwartz function~$\Phi$ and apply Proposition~\ref{GrandMaximalAtom}.
\end{proof}

\begin{cor}\label{PointwiseMaximalAtom}
Let~$a\in \mathcal{S}'(\R^d)$ be a~$(\beta,N)$-atom adapted to the cube~$Q$. Then, for some collection~$\F$,
\eq{
\M_{\F,d-\beta}[a](x) \lesssim (\ell(Q))^{-\beta}\Big(1 + \frac{\dist(x,Q)}{\ell(Q)}\Big)^{-N-1-\beta}.
}
\end{cor}

\begin{lem}\label{maximal bound}
If~$a$ is a~$\beta$-atom, then
\begin{align*}
\|\M_{\F,d-\beta} a\|_{L^1(\Haus_\infty^\beta)} \lesssim 1.
\end{align*}
\end{lem}

\begin{proof}[Proof of Lemma \ref{maximal bound}]
We wish to bound the sum
\eq{
\sum\limits_{k \geq 0}2^{-k} \Haus_\infty^\beta\Big(\Set{x\in\R^d}{\M_{\F,d-\beta}[a](x) \geq 2^{-k}}\Big).
}
By Corollary~\ref{PointwiseMaximalAtom}, the latter sum is bounded by
\eq{
\sum\limits_{k \geq 0} 2^{-k} 2^{\frac{k\beta}{\beta+1}} \lesssim 1.
}
\end{proof}

\section{The Dimension Estimate Revisited}\label{dimension_estimates}

We begin this Section with the
\begin{proof}[Proof of Theorem \ref{dimension complicated}]

We follow the proof of Theorem \ref{dimension simplified}, mutatis mutandis.  

As there, it suffices to prove that for $A \subset Q_0$, $\mathcal{H}^\beta(A)=0$ implies $|\mu|(A)=0$ and we show that for all~$\epsilon > 0 $ there exists~$\delta > 0$ such that for any Borel set~$A \subset Q_0$ with~$\Haus^{\beta}_\infty(A) < \delta$ we also have~$|\mu|(A) < \epsilon$.

Let $\epsilon>0$ be given.  First, an argument analogous to that of claim \eqref{absolute continuity}  shows that there exists $\delta>0$ such that if $\Haus^{\beta}_\infty(A) < \delta$, then
\begin{align}\label{absolute continuity grand maximal}
\int_A \mathrm{M}_{\mathcal{F},d-\beta} [\mu] \;d\mathcal{H}^{\beta}_\infty <c\epsilon
\end{align}
for any fixed choice of $c>0$ to be chosen later (and which is only for aesthetic purposes).

Second, as in \eqref{small_cube_control} we may choose $N_\epsilon \in \mathbb{N}$ such that
\begin{align*}
|\mu|( \{ x \in Q_0 : 0<l_x<\frac{1}{n}\}) <c\epsilon 
\end{align*}
for every $n\geq N_\epsilon$.  Then the set
\begin{align*}
A':=\{ x \in A : l_x \geq \frac{1}{N_\epsilon}\}
\end{align*}
satisfies
\begin{align*}
    |\mu|(Q) \leq 2|\mu(Q)|
\end{align*}
for $x \in A'$
for all cubes $Q$ which contain $x$ and satisfy  $l(Q) < \frac{1}{N_\epsilon}$.  Therefore we find the analogue of \eqref{large_cubes_remain}:
\begin{align}
|\mu|(A) \leq |\mu|(A') + c\epsilon.
\end{align}

Let $A$ be a Borel set such that with~$\Haus_\infty^{\beta}(A) < \delta$. 
 Then one can find a balls $\{B_{r_j}(x_j)\}$  such that~$A \subset \cup_{j} B_{r_j}(x_j)$ and
\eq{
\sum\limits_{j}r_j^\beta < \delta.
}
Without loss of generality, we may assume~$A\cap B_{r_{j}}(x_j)\ne \varnothing$ for any~$j$. We may also assume that the sequence~$\{r_j\}_j$ is non-increasing.

Let now~$A_n = \cup_{j=1}^nB_{r_j}(x_j)$, here~$n\in \N\cup\infty$. The set~$A_\infty$ is bounded, and, therefore, has finite measure~$|\mu|$. Thus, there exists~$n$ such that~$|\mu|(A_\infty \setminus A_n) < c\epsilon$, and so
\begin{align}
|\mu|(A') &\leq |\mu|(A'\setminus A_n) + |\mu|(A'\cap A_n)\\
&\leq |\mu|(A_\infty\setminus A_n) + |\mu|(A'\cap A_n)\\
&\leq c\epsilon+ |\mu|(A'\cap A_n).
\end{align}
Note that since~$\Haus_\infty^\beta(A_n) < \delta$ by \eqref{absolute continuity grand maximal}
\eq{
\int\limits_{A_n}\MM_{\alpha, \F, r_n}[\mu]d\Haus_\infty^\beta \leq \int\limits_{A_n}\MM_{\alpha, \F}[\mu]d\Haus_\infty^\beta < c\epsilon,
} 
where~$\MM_{\alpha,\F,r_n}$ is the anti-local maximal function defined by
\eq{
\MM_{\alpha,\F, \rho}[\mu] = \sup\limits_{\Phi \in S_\F} \sup\limits_{t > \rho}t^{\alpha}\big|\mu*\Phi_t(x)\big|.
} 
By now we have only used the obvious inequality~$\MM_{\alpha,\F,r_n}[\mu] \leq \MM_{\alpha,\F}[\mu]$; below we will use that the truncated function is somehow regular; this property is expressed by the following claim.

{\bf Claim.} Let~$\MM_{\alpha,\F,\rho}[\mu](x) > \lambda$ for some point~$x\in\R^d$ and~$\lambda > 0$. Then
\eq{
\MM_{\alpha,\F,\rho}[\mu](y) \gtrsim \lambda,\qquad y\in B_{\rho}(x).
}
Let us prove the claim. Assume
\eq{
t^{\alpha}|\mu*\Phi_t(x)| \geq \lambda,\qquad\text{for some}\ t > \rho, \ \|\Phi\|_{\F} \leq 1.
} 
Then, we have
\eq{
t^{\alpha} |\mu*\Phi_t(y)| = t^{\alpha}|\mu*\Psi_t(x)|, 
}
where the function~$\Psi$ is defined by the formula~$\Psi(z)=\Phi(z + (y-x)/t)$. Since~$y\in B_\rho(x)$ and~$t > \rho$, we have~$|(y-x)/t| \leq 1$, and~$\|\Psi\|_{\F} \lesssim 1$. {\bf This proves the claim. }

Now consider the set~$A_n$, which is the union of a finite number of balls of radii at least~$r_n$. Denote by $\cup_j U_j$ a cover of $A_n$ by balls of radius $r_n$ with finite overlap. Then
\begin{align*}
A_n &\subset \cup_j U_j\\
\cup_j U_j &\subset \cup_{j=1}^nB_{2r_j}(x_j)=: \tilde{A}_n.
\end{align*}
Then the fact that
\begin{align*}
\mathcal{H}^\beta_\infty(\tilde{A}_n) \leq  2^\beta \delta
\end{align*}
and by choosing $\delta$ slightly smaller implies therefore
\begin{align*}
\int\limits_{\tilde{A}_n}\MM_{\alpha, \F, r_n}[\mu]\;d\Haus_\infty^\beta <c\epsilon.
\end{align*}

For $c_j$ the center of the ball~$U_j$ of radius $r_n$, consider the sets
\begin{align}
\Omega_{n,k} = \cup_j U_j,\quad \text{where}\ \M_{\alpha,\F, r_n}[\mu](c_j) \in [2^k,2^{k+1}). 
\end{align}
By the claim,~$\M_{\alpha,\F,r_n}[\mu] \sim 2^k$ on~$\Omega_{n,k}$, and each~$\Omega_{n,k}$ is contained in~$\tilde{A}_n$. Thus,
\eq{\label{eq215}
\sum\limits_k 2^{k}\Haus_\infty^\beta(\Omega_{n,k})< c\epsilon.
}
By Corollary~\ref{StructuralCorollary} (packing condition lemma), there is a covering~$\{B_{n,k,j}\}_j$ of~$\Omega_{n,k}$ by the balls of radii at least~$r_n$ such that the sum of their radii raised to the power~$\beta$ is controlled by~$\Haus_\infty^{\beta}(\Omega_{n,k})$:
\begin{align*}
\sum_j r_{n,k,j}^\beta \leq C\Haus_\infty^{\beta}(\Omega_{n,k}).
\end{align*}

As we wish to estimate~$|\mu|(A'\cap A_n)$, we may also assume each set~$B_{n,k,j}$ covers at least one point~$x_{n,k,j}$ of~$A'$, and by doubling the radii that this point is the center of $B_{n,k,j}$. Then,
\mlt{
|\mu| (A'\cap A_n) \leq \sum\limits_{k,j}|\mu|(B_{n,k,j}) \leq \sum\limits_{k,j}r_{n,k,j}^{d} |\mu|*\Phi_{r_{n,k,j}}(x_{n,k,j})\\ \Lref{\scriptscriptstyle r_{n,k,j} \leq 1/N_\epsilon} 2\sum\limits_{k,j}\left| r_{n,k,j}^{d} \mu*\Phi_{r_{n,k,j}}(x_{n,k,j}) \right|\leq \sum\limits_{k,j}2^{k} r_{n,k,j}^\beta \Leqref{eq215} 2Cc\epsilon.
}
Here~$\Phi$ is a radially symmetric radially decreasing Schwartz function with compact support,~$r_{n,k,j}$ is the radius of~$Q_{n,k,j}$, and~$x_{n,k,j}$ is its center.

This shows that
\begin{align*}
|\mu|(A) \leq |\mu|(A') + c\epsilon \\
&\leq  |\mu|(A'\cap A_n)+ 2c\epsilon\\
&\leq (2C+2)c\epsilon,\\
&\leq \epsilon
\end{align*}
with the choice $c=(2C+2)^{-1}$, where $C>0$ is the constant arising in the packing condition Corollary~\ref{StructuralCorollary}.  In particular, for any $\epsilon>0$ one has
\begin{align*}
|\mu|(A) \leq \epsilon
\end{align*}
whenever $\Haus_\infty^\beta(A) < \delta$ for $\delta>0$ small enough that $\Haus_\infty^\beta(A) < 2^\beta \delta$ implies
\begin{align*}
\int_A \mathrm{M}_{\mathcal{F},d-\beta} [\mu] \;d\mathcal{H}^{\beta}_\infty <c\epsilon
\end{align*}
for the fixed choice of $\epsilon>0$ and $c=(2C+2)^{-1}$.

\end{proof}

\section*{Acknowledgements}
D. Spector is supported by the National Science and Technology Council of Taiwan under research grant number 110-2115-M-003-020-MY3 and the Taiwan Ministry of Education under the Yushan Fellow Program.  D. Stolyarov is supported by the Basis Foundation grant no. 21-7-2-12-1.
%


\begin{bibdiv}

\begin{biblist}


\bib{Adams:1988}{article}{
   author={Adams, D. R.},
   title={A note on Choquet integrals with respect to Hausdorff capacity},
   conference={
      title={Function spaces and applications},
      address={Lund},
      date={1986},
   },
   book={
      series={Lecture Notes in Math.},
      volume={1302},
      publisher={Springer, Berlin},
   },
   date={1988},
   pages={115--124},
}
\bib{AAR}{article}{
	title={An elementary approach to the dimension of measures satisfying a first-order linear PDE constraint},
	author={Arroyo-Rabasa, A.},
	journal={Proc. Amer. Math. Soc.},
	volume={148},
	number={1},
	pages={273--282},
	year={2020}
}

\bib{ARDPHR}{article}{
   author={Arroyo-Rabasa, Adolfo},
   author={De Philippis, Guido},
   author={Hirsch, Jonas},
   author={Rindler, Filip},
   title={Dimensional estimates and rectifiability for measures satisfying
   linear PDE constraints},
   journal={Geom. Funct. Anal.},
   volume={29},
   date={2019},
   number={3},
   pages={639--658},
   issn={1016-443X},
   review={\MR{3962875}},
   doi={10.1007/s00039-019-00497-1},
}




\bib{RA}{article}{
   author={Ayoush, Rami},
   title={On finite configurations in the spectra of singular measures},
   journal={Math. Z.},
   volume={304},
   date={2023},
   number={1},
   pages={Paper No. 6, 17},
   issn={0025-5874},
   review={\MR{4569405}},
   doi={10.1007/s00209-023-03257-y},
}

\bib{AW}{article}{
   author={Ayoush, Rami},
   author={Wojciechowski, Micha\l },
   title={On dimension and regularity of vector-valued measures under
   Fourier analytic constraints},
   journal={Illinois J. Math.},
   volume={66},
   date={2022},
   number={3},
   pages={289--313},
   issn={0019-2082},
   review={\MR{4484222}},
   doi={10.1215/00192082-10018154},
}

\bib{ASW}{article}{
AUTHOR = {R. Ayoush and D. Stolyarov and  M. Wojciechowski},
     TITLE = {Sobolev martingales},
   JOURNAL = {Rev. Mat. Iberoam.},
  FJOURNAL = {Revista Matem\'{a}tica Iberoamericana},
    VOLUME = {37},
      YEAR = {2021},
    NUMBER = {4},
     PAGES = {1225--1246},
      ISSN = {0213-2230,2235-0616},
   MRCLASS = {60G46 (28A80 46E35)},
  MRNUMBER = {4269395},
       DOI = {10.4171/rmi/1224},
       URL = {https://doi.org/10.4171/rmi/1224},
}


\bib{BourgainBrezis2004}{article}{
   author={Bourgain, Jean},
   author={Brezis, Ha\"{\i}m},
   title={New estimates for the Laplacian, the div-curl, and related Hodge
   systems},
   language={English, with English and French summaries},
   journal={C. R. Math. Acad. Sci. Paris},
   volume={338},
   date={2004},
   number={7},
   pages={539--543},
   issn={1631-073X},
   review={\MR{2057026}},
   doi={10.1016/j.crma.2003.12.031},
}
\bib{BourgainBrezis2007}{article}{
   author={Bourgain, Jean},
   author={Brezis, Ha\"{\i}m},
   title={New estimates for elliptic equations and Hodge type systems},
   journal={J. Eur. Math. Soc. (JEMS)},
   volume={9},
   date={2007},
   number={2},
   pages={277--315},
   issn={1435-9855},
   review={\MR{2293957}},
   doi={10.4171/JEMS/80},
}


\bib{Chen-Spector}{article}{
   author={Chen, You-Wei},
   author={Spector, Daniel},
   title={On functions of bounded $\beta$-dimensional mean oscillation},
   journal={Adv. Calc. Var.},
doi = {doi:10.1515/acv-2022-0084},
url = {https://doi.org/10.1515/acv-2022-0084},
}

\bib{Dorronsoro}{article}{
   author={Dorronsoro, Jos\'{e} R.},
   title={Differentiability properties of functions with bounded variation},
   journal={Indiana Univ. Math. J.},
   volume={38},
   date={1989},
   number={4},
   pages={1027--1045},
   issn={0022-2518},
   review={\MR{1029687}},
   doi={10.1512/iumj.1989.38.38047},
}





\bib{Falconer}{book}{
   author={Falconer, Kenneth},
   title={Techniques in fractal geometry},
   publisher={John Wiley \& Sons, Ltd., Chichester},
   date={1997},
   pages={xviii+256},
   isbn={0-471-95724-0},
   review={\MR{1449135}},
}





\bib{FF}{article}{
   author={Federer, Herbert},
   author={Fleming, Wendell H.},
   title={Normal and integral currents},
   journal={Ann. of Math. (2)},
   volume={72},
   date={1960},
   pages={458--520},
   issn={0003-486X},
   review={\MR{0123260}},
   doi={10.2307/1970227},
}


\bib{FeffermanStein}{article}{
   author={Fefferman, C.},
   author={Stein, E. M.},
   title={$H^{p}$ spaces of several variables},
   journal={Acta Math.},
   volume={129},
   date={1972},
   number={3-4},
   pages={137--193},
   issn={0001-5962},
   review={\MR{447953}},
   doi={10.1007/BF02392215},
}




\bib{Gagliardo}{article}{
   author={Gagliardo, Emilio},
   title={Propriet\`a di alcune classi di funzioni in pi\`u variabili},
   language={Italian},
   journal={Ricerche Mat.},
   volume={7},
   date={1958},
   pages={102--137},
   issn={0035-5038},
   review={\MR{102740}},
}

\bib{GS}{article}{
   author={Garg, Rahul},
   author={Spector, Daniel},
   title={On the regularity of solutions to Poisson's equation},
   journal={C. R. Math. Acad. Sci. Paris},
   volume={353},
   date={2015},
   number={9},
   pages={819--823},
   issn={1631-073X},
   review={\MR{3377679}},
   doi={10.1016/j.crma.2015.07.001},
}

\bib{GS1}{article}{
   author={Garg, Rahul},
   author={Spector, Daniel},
   title={On the role of Riesz potentials in Poisson's equation and Sobolev
   embeddings},
   journal={Indiana Univ. Math. J.},
   volume={64},
   date={2015},
   number={6},
   pages={1697--1719},
   issn={0022-2518},
   review={\MR{3436232}},
   doi={10.1512/iumj.2015.64.5706},
}
\bib{grafakos}{book}{
   author={Grafakos, Loukas},
   title={Classical Fourier analysis},
   series={Graduate Texts in Mathematics},
   volume={249},
   edition={3},
   publisher={Springer, New York},
   date={2014},
   pages={xviii+638},
}


\bib{SpectorHernandezRaita2021}{article}{
	author = {Hernandez, F.},
	author ={Spector, D.},
	author ={Rai\cb{t}\u{a}, B.},
	title = {Endpoint $L^1$
		estimates for Hodge systems}
journal={Math. Ann.}
	note={ https://doi.org/10.1007/s00208-022-02383-y}
		date={2022}
}

\bib{GHS}{article}{
	author = {Hernandez, F.},
	author ={Spector, D.},
	author ={Goodman, J.},
	title = {Two approximation results for divergence free measures},
journal={Port. Math.},
	note={ https://arxiv.org/abs/2010.14079},
		date={to appear}
}

\bib{HS}{article}{
	author = {Hernandez, F.},
	author ={Spector, D.},
	title = {Fractional Integration and Optimal Estimates for Elliptic Systems},
	journal={Calc. Var. Partial Differential Equations},
	note = {https://arxiv.org/abs/2008.05639},
	date={to appear}
}


\bib{MS}{article}{
   author={Mart\'{\i}nez, \'{A}ngel D.},
   author={Spector, Daniel},
   title={An improvement to the John-Nirenberg inequality for functions in
   critical Sobolev spaces},
   journal={Adv. Nonlinear Anal.},
   volume={10},
   date={2021},
   number={1},
   pages={877--894},
   issn={2191-9496},
   review={\MR{4191703}},
   doi={10.1515/anona-2020-0157},
}




\bib{Mattila}{book}{
   author={Mattila, Pertti},
   title={Geometry of sets and measures in Euclidean spaces},
   series={Cambridge Studies in Advanced Mathematics},
   volume={44},
   note={Fractals and rectifiability},
   publisher={Cambridge University Press, Cambridge},
   date={1995},
   pages={xii+343},
   isbn={0-521-46576-1},
   isbn={0-521-65595-1},
   review={\MR{1333890}},
   doi={10.1017/CBO9780511623813},
}

\bib{mazya}{article}{
   author={Maz\cprime ja, V. G.},
   title={Classes of domains and imbedding theorems for function spaces},
      journal={Soviet Math. Dokl.},
      volume={1},
      date={1960},
      issn={0197-6788},
   translation={
      journal={Soviet Math. Dokl.},
      volume={1},
      date={1960},
      pages={882--885},
      issn={0197-6788},
   },
   review={\MR{0126152}},
}


\bib{maz_trace}{article}{
   author={Maz\cprime ja, V. G.},
   title={Strong capacity-estimates for ``fractional'' norms},
   note={Numerical methods and questions on organization of computations},
   language={Russian},
   journal={Zap. Nau\v{c}n. Sem. Leningrad. Otdel. Mat. Inst. Steklov.
   (LOMI)},
   volume={70},
   date={1977},
   pages={161--168, 292},
   review={\MR{0500118}},
}



\bib{Meyers-Ziemer}{article}{
      author={Meyers, N.~G.},
      author={Ziemer, W.~P.},
       title={Integral inequalities of {P}oincar\'e and {W}irtinger type for
  {BV} functions},
        date={1977},
     journal={Amer. J. Math.},
      volume={99},
       pages={1345\ndash 1360},
}


\bib{Nirenberg}{article}{
   author={Nirenberg, L.},
   title={On elliptic partial differential equations},
   journal={Ann. Scuola Norm. Sup. Pisa Cl. Sci. (3)},
   volume={13},
   date={1959},
   pages={115--162},
   issn={0391-173X},
   review={\MR{109940}},
}

\bib{Ponce-Spector}{article}{
   author={Ponce, Augusto C.},
   author={Spector, Daniel},
   title={A boxing inequality for the fractional perimeter},
   journal={Ann. Sc. Norm. Super. Pisa Cl. Sci. (5)},
   volume={20},
   date={2020},
   number={1},
   pages={107--141},
   issn={0391-173X},
   review={\MR{4088737}},
}
\bib{Ponce-Spector-2}{article}{
   author={Ponce, Augusto C.},
   author={Spector, Daniel},
   title={A decomposition by non-negative functions in the Sobolev space
   $W^{k,1}$},
   journal={Indiana Univ. Math. J.},
   volume={69},
   date={2020},
   number={1},
   pages={151--169},
   issn={0022-2518},
   review={\MR{4077159}},
   doi={10.1512/iumj.2020.69.8237},
}

\bib{Ponce-Spector-1}{article}{
   author={Ponce, Augusto C.},
   author={Spector, Daniel},
   title={Some remarks on capacitary integrals and measure theory},
   conference={
      title={Potentials and partial differential equations---the legacy of
      David R. Adams},
   },
   book={
      series={Adv. Anal. Geom.},
      volume={8},
      publisher={De Gruyter, Berlin},
   },
   isbn={978-3-11-079265-2},
   isbn={978-3-11-079272-0},
   isbn={978-3-11-079278-2},
   date={[2023] \copyright 2023},
   pages={235--263},
   review={\MR{4654520}},
}

\bib{Raita_report}{article}{
	title={L1-estimates and A-weakly differentiable functions},
	author={Rai{\cb{t}}{\u{a}}, B.},
	year={2018},
	note={Technical Report OxPDE-18/01, University of Oxford}
}



\bib{RSS}{article}{
   author={Rai\c{t}\u{a}, Bogdan},
   author={Spector, Daniel},
   author={Stolyarov, Dmitriy},
   title={A trace inequality for solenoidal charges},
   journal={Potential Anal.},
   volume={59},
   date={2023},
   number={4},
   pages={2093--2104},
   issn={0926-2601},
   review={\MR{4684387}},
   doi={10.1007/s11118-022-10008-x},
}

\bib{RW}{article}{
AUTHOR = {M. Roginskaya and M. Wojciechowski},
     TITLE = {Singularity of vector valued measures in terms of {F}ourier
              transform},
   JOURNAL = {J. Fourier Anal. Appl.},
  FJOURNAL = {The Journal of Fourier Analysis and Applications},
    VOLUME = {12},
      YEAR = {2006},
    NUMBER = {2},
     PAGES = {213--223},
      ISSN = {1069-5869,1531-5851},
   MRCLASS = {42B10 (28A80)},
  MRNUMBER = {2224396},
MRREVIEWER = {Xavier\ Tolsa},
       DOI = {10.1007/s00041-005-5030-9},
       URL = {https://doi.org/10.1007/s00041-005-5030-9}
}

\bib{Sawyer1}{article}{
   author={Sawyer, Eric T.},
   title={Two weight norm inequalities for certain maximal and integral
   operators},
   conference={
      title={Harmonic analysis},
      address={Minneapolis, Minn.},
      date={1981},
   },
   book={
      series={Lecture Notes in Math.},
      volume={908},
      publisher={Springer, Berlin-New York},
   },
   date={1982},
   pages={102--127},
   review={\MR{654182}},
}

		
\bib{Sawyer2}{article}{
   author={Sawyer, Eric T.},
   title={Weighted norm inequalities for fractional maximal operators},
   conference={
      title={1980 Seminar on Harmonic Analysis},
      address={Montreal, Que.},
      date={1980},
   },
   book={
      series={CMS Conf. Proc.},
      volume={1},
      publisher={Amer. Math. Soc., Providence, R.I.},
   },
   date={1981},
   pages={283--309},
   review={\MR{670111}},
}



\bib{SSVS}{article}{
   author={Schikorra, Armin},
   author={Spector, Daniel},
   author={Van Schaftingen, Jean},
   title={An $L^1$-type estimate for Riesz potentials},
   journal={Rev. Mat. Iberoam.},
   volume={33},
   date={2017},
   number={1},
   pages={291--303},
   issn={0213-2230},
   review={\MR{3615452}},
   doi={10.4171/RMI/937},
}

\bib{Smirnov}{article}{
    AUTHOR = {Smirnov, S. K.},
     TITLE = {Decomposition of solenoidal vector charges into elementary
              solenoids, and the structure of normal one-dimensional flows},
   JOURNAL = {Algebra i Analiz},
  FJOURNAL = {Rossi\u{\i}skaya Akademiya Nauk. Algebra i Analiz},
    VOLUME = {5},
      YEAR = {1993},
    NUMBER = {4},
     PAGES = {206--238},
      ISSN = {0234-0852},
   MRCLASS = {49Q15 (49Q20 58A25 58C35)},
  MRNUMBER = {1246427},
MRREVIEWER = {Andrew\ Bucki}
}



\bib{Sobolev}{article}{
   author={Sobolev, S.L.},
    title={On a theorem of functional analysis},
   journal={Mat. Sb.},
   volume={4},
   number={46},
  year={1938},
  language={Russian},
   pages={471-497},
   translation={
      journal={Transl. Amer. Math. Soc.},
      volume={34},
     date={},
      pages={39-68},
   },
   }
   
   \bib{Spector}{article}{
   author={Spector, Daniel},
   title={An optimal Sobolev embedding for $L^1$},
   journal={J. Funct. Anal.},
   volume={279},
   date={2020},
   number={3},
   pages={108559, 26},
   issn={0022-1236},
   review={\MR{4093790}},
   doi={10.1016/j.jfa.2020.108559},
}

   \bib{Spector-NA}{article}{
   author={Spector, Daniel},
   title={New directions in harmonic analysis on $L^1$},
   journal={Nonlinear Anal.},
   volume={192},
   date={2020},
   pages={111685, 20},
   issn={0362-546X},
   review={\MR{4034690}},
   doi={10.1016/j.na.2019.111685},
}

\bib{Spector-PM}{article}{
   author={Spector, Daniel},
   title={A noninequality for the fractional gradient},
   journal={Port. Math.},
   volume={76},
   date={2019},
   number={2},
   pages={153--168},
   issn={0032-5155},
   review={\MR{4065096}},
   doi={10.4171/pm/2031},
}

\bib{SVS}{article}{
   author={Spector, Daniel},
   author={Van Schaftingen, Jean},
   title={Optimal embeddings into Lorentz spaces for some vector
   differential operators via Gagliardo's lemma},
   journal={Atti Accad. Naz. Lincei Rend. Lincei Mat. Appl.},
   volume={30},
   date={2019},
   number={3},
   pages={413--436},
   issn={1120-6330},
   review={\MR{4002205}},
   doi={10.4171/RLM/854},
}

\bib{S}{book}{
   author={Stein, Elias M.},
   title={Singular integrals and differentiability properties of functions},
   series={Princeton Mathematical Series, No. 30},
   publisher={Princeton University Press, Princeton, N.J.},
   date={1970},
   pages={xiv+290},
   review={\MR{0290095}},
}

\bib{Stein93}{book}{
   author={Stein, Elias M.},
   title={Harmonic analysis: real-variable methods, orthogonality, and
   oscillatory integrals},
   series={Princeton Mathematical Series},
   volume={43},
   note={With the assistance of Timothy S. Murphy;
   Monographs in Harmonic Analysis, III},
   publisher={Princeton University Press, Princeton, NJ},
   date={1993},
   pages={xiv+695},
   isbn={0-691-03216-5},
   review={\MR{1232192}},
}

   \bib{Stein-Weiss}{article}{
   author={Stein, Elias M.},
   author={Weiss, Guido},
   title={On the theory of harmonic functions of several variables. I. The
   theory of $H\sp{p}$-spaces},
   journal={Acta Math.},
   volume={103},
   date={1960},
   pages={25--62},
   issn={0001-5962},
   review={\MR{121579}},
   doi={10.1007/BF02546524},
}















\bib{Stol-Woj}{article}{
   author={Stolyarov, Dmitriy M.},
   author={Wojciechowski, Michal},
   title={Dimension of gradient measures},
   language={English, with English and French summaries},
   journal={C. R. Math. Acad. Sci. Paris},
   volume={352},
   date={2014},
   number={10},
   pages={791--795},
   issn={1631-073X},
   review={\MR{3262909}},
   doi={10.1016/j.crma.2014.08.011},
}

   \bib{Stolyarov}{article}{
   author={Stolyarov, D. M.},
   title={Hardy-Littlewood-Sobolev inequality for $p=1$},
   language={Russian, with Russian summary},
   journal={Mat. Sb.},
   volume={213},
   date={2022},
   number={6},
   pages={125--174},
   issn={0368-8666},
   review={\MR{4461456}},
   doi={10.4213/sm9645},
}

\bib{Stolyarov-1}{article}{
   author={Stolyarov, D. M.},
   title={A theorem of Dorronsoro and its slight generalization},
   language={Russian, with English summary},
   journal={Zap. Nauchn. Sem. S.-Peterburg. Otdel. Mat. Inst. Steklov.
   (POMI)},
   volume={434},
   date={2015},
   pages={126--135},
   issn={0373-2703},
   translation={
      journal={J. Math. Sci. (N.Y.)},
      volume={215},
      date={2016},
      number={5},
      pages={624--630},
      issn={1072-3374},
   },
   review={\MR{3493705}},
   doi={10.1007/s10958-016-2869-z},
}

\bib{DS}{article}{
   author={Stolyarov, Dmitriy},
   title={Dimension estimates for vectorial measures with restricted
   spectrum},
   journal={J. Funct. Anal.},
   volume={284},
   date={2023},
   number={1},
   pages={Paper No. 109735},
   issn={0022-1236},
   review={\MR{4500728}},
   doi={10.1016/j.jfa.2022.109735},
}



   
\bib{Strauss}{article}{
   author={Strauss, Monty J.},
   title={Variations of Korn's and Sobolev's equalities},
   conference={
      title={Partial differential equations},
      address={Proc. Sympos. Pure Math., Vol. XXIII, Univ. California,
      Berkeley, Calif.},
      date={1971},
   },
   book={
      publisher={Amer. Math. Soc., Providence, R.I.},
   },
   date={1973},
   pages={207--214},
   review={\MR{0341064}},
}



\bib{Uchiyama}{article}{
   author={Uchiyama, Akihito},
   title={A constructive proof of the Fefferman-Stein decomposition of BMO
   $({\bf R}\sp{n})$},
   journal={Acta Math.},
   volume={148},
   date={1982},
   pages={215--241},
   issn={0001-5962},
   review={\MR{0666111}},
   doi={10.1007/BF02392729},
}

\bib{VS}{article}{
   author={Van Schaftingen, Jean},
   title={Function spaces between BMO and critical Sobolev spaces},
   journal={J. Funct. Anal.},
   volume={236},
   date={2006},
   number={2},
   pages={490--516},
   issn={0022-1236},
   review={\MR{2240172}},
   doi={10.1016/j.jfa.2006.03.011},
}

\bib{VanSchaftingen_2004}{article}{
   author={Van Schaftingen, Jean},
   title={Estimates for $L^1$--vector fields},
   journal={C. R. Math. Acad. Sci. Paris},
   volume={339},
   date={2004},
   number={3},
   pages={181--186},
   issn={1631-073X},
   doi={10.1016/j.crma.2004.05.013},
}

\bib{VanSchaftingen_2004_ARB}{article}{
   author={Van Schaftingen, Jean},
   title={Estimates for $L^1$ vector fields with a second order condition},
   journal={Acad. Roy. Belg. Bull. Cl. Sci. (6)},
   volume={15},
   date={2004},
   number={1-6},
   pages={103--112},
   issn={0001-4141},
}
   


\bib{VanSchaftingen_2010}{article}{
   author={Van Schaftingen, Jean},
   title={Limiting fractional and Lorentz space estimates of differential
   forms},
   journal={Proc. Amer. Math. Soc.},
   volume={138},
   date={2010},
   number={1},
   pages={235--240},
   issn={0002-9939},
   doi={10.1090/S0002-9939-09-10005-9},
}
    
\bib{VanSchaftingen_2013}{article}{
   author={Van Schaftingen, Jean},
   title={Limiting Sobolev inequalities for vector fields and canceling
   linear differential operators},
   journal={J. Eur. Math. Soc. (JEMS)},
   volume={15},
   date={2013},
   number={3},
   pages={877--921},
   issn={1435-9855},
   doi={10.4171/JEMS/380},
}

\bib{VanSchaftingen_2015}{article}{
   author={Van Schaftingen, Jean},
   title={Limiting Bourgain-Brezis estimates for systems of linear
   differential equations: theme and variations},
   journal={J. Fixed Point Theory Appl.},
   volume={15},
   date={2014},
   number={2},
   pages={273--297},
   issn={1661-7738},
   doi={10.1007/s11784-014-0177-0},
}     
\bib{Zygmund}{article}{
   author={Zygmund, A.},
   title={On a theorem of Marcinkiewicz concerning interpolation of
   operations},
   journal={J. Math. Pures Appl. (9)},
   volume={35},
   date={1956},
   pages={223--248},
   issn={0021-7824},
   review={\MR{80887}},
}







 



\end{biblist}
	
\end{bibdiv}
\end{document}